\documentclass[11pt,reqno]{amsart}
\usepackage[margin=1in]{geometry}
\usepackage{hyperref, amssymb, amsfonts, amsthm, amsrefs, mathtools, enumerate, verbatim, color, makecell, nicefrac, bm, stmaryrd, microtype}

\newtheorem{Theorem}{Theorem}[section]
\newtheorem{Proposition}[Theorem]{Proposition}
\newtheorem{Lemma}[Theorem]{Lemma}

\newtheorem*{Claim}{Claim}

\theoremstyle{definition}
\newtheorem{Definition}[Theorem]{Definition}

\newtheorem{remark}[Theorem]{Remark}

\newcommand{\rca}{\mathsf{RCA}_0}
\newcommand{\wkl}{\mathsf{WKL}_0}
\newcommand{\aca}{\mathsf{ACA}_0}
\newcommand{\pa}{\mathsf{PA}}

\newcommand{\pica}{\Pi^1_1\mbox{-}\mathsf{CA}_0}

\DeclareMathOperator{\dom}{\mathrm{dom}}
\DeclareMathOperator{\ran}{\mathrm{ran}}

\newcommand{\lscdom}{{\rm dom}}
\newcommand{\lscsupp}{{\rm supp}}

\newcommand{\andd}{\wedge}

\newcommand{\la}{\langle}
\newcommand{\ra}{\rangle}
\newcommand{\da}{{\downarrow}}

\newcommand{\imp}{\rightarrow}
\newcommand{\Imp}{\Rightarrow}
\newcommand{\biimp}{\leftrightarrow}

\newcommand{\smf}{\smallfrown}
\newcommand{\Nb}{\mathbb{N}}
\newcommand{\Qb}{\mathbb{Q}}
\newcommand{\Rb}{\mathbb{R}}

\newcommand{\fnc}{f}

\newcommand{\fvp}{\mathrm{FVP}}
\newcommand{\lvp}{\mathrm{LVP}}
\usepackage[inline]{enumitem}

\newcommand{\sel}{(}
\newcommand{\ser}{)}

\newcommand{\ucode}[1]{\stackrel {#1}\rightharpoonup}
\newcommand{\vcode}[1]{\stackrel {#1}\rightharpoondown}
\newcommand{\fcode}[1]{\stackrel {#1}\to}
\newcommand{\ballsub}{\subsetplus}

\newcommand{\secsub}{\sqsubset}
\newcommand{\secseq}{\sqsubseteq}
\newcommand{\BaireSet}{\Nb^{<\Nb}}
\newcommand{\peq}{\preccurlyeq}
\newcommand{\ignore}[1]{}
\def\compax{{\rm CA}}

\newcommand{\ball}[2]{B_{#2}(#1)}

\newcommand{\rest}{{\upharpoonright}}

\newcommand{\MC}[1]{\mathcal{#1}}
\newcommand{\spc}[1]{\mathcal{#1}}

\usepackage{xcolor}	
\usepackage{soul}

\AtBeginDocument{%
   \def\MR#1{}
}

\title[Ekeland's principle in weak and strong systems]{Ekeland's variational principle in weak and strong systems of arithmetic}

\author{David Fern\'andez-Duque}
\address{Department of Mathematics\\
Ghent University\\
Krijgslaan 281 S22\\
B-9000 Ghent\\
Belgium}
\email{David.FernandezDuque@UGent.be}

\author{Paul Shafer}
\address{
School of Mathematics\\
University of Leeds\\
Leeds\\
LS2 9JT\\
UK}
\email{p.e.shafer@leeds.ac.uk}
\urladdr{\url{http://www1.maths.leeds.ac.uk/~matpsh/}}

\author{Keita Yokoyama}
\address {School of Information Science\\ Japan Advanced Institute of Science and Technology\\
1-1 Asahidai, Nomi, Ishikawa 923-1292, Japan}
\email{y-keita@jaist.ac.jp}
\urladdr{\url{http://www.jaist.ac.jp/~y-keita/}}



\date{\today}

\begin{document}

\begin{abstract}
We analyze Ekeland's variational principle in the context of reverse mathematics. We find that that the full variational principle is equivalent to $\pica$, a strong theory of second-order arithmetic, while natural restrictions (e.g.~to compact spaces or to continuous functions) yield statements equivalent to weak K\"onig's lemma ($\wkl$) and to arithmetical comprehension ($\aca$).  We also find that the localized version of Ekeland's variational principle is equivalent to $\pica$, even when restricted to continuous functions.  This is a rare example of a statement about continuous functions having great logical strength.
\end{abstract}

\maketitle

\section{Introduction}

The field of {\em reverse mathematics}, introduced by H.\ Friedman~\cite{Friedman}, aims to identify the minimal foundational assumptions required to prove specific results from several mathematical fields, including mathematical analysis.
There are many advantages to determining these minimal assumptions, among them aiding in extracting computational content from theorems whose original proofs were non-constructive.
It is also desirable from a methodological perspective to avoid strong set-theoretic assumptions when possible, and indeed advances in reverse mathematics have shown us that a large portion of known mathematics can be carried out within a relatively small fragment of second-order arithmetic \cite{SimpsonSOSOA}.
Second-order arithmetic extends the language of Peano arithmetic by adding variables and quantifiers ranging over sets of natural numbers (see \S\ref{SecSSOA}); this is sufficient to formalize many familiar concepts from analysis, including real numbers (and, more generally, points in a complete separable metric space, also called a {\em Polish space}), continuous functions, and open and closed sets (see \S\ref{sec-MetricDefs}).
One can then ask which axioms are needed to prove (say) the Stone--Weierstrass theorem in this framework.

Although there are many exceptions, a surprisingly large portion of the theorems analyzed are equivalent to one of a handful of systems of second-order arithmetic. These axiomatic systems are always assumed to extend the `base theory' of reverse mathematics, {\em recursive comprehension} ($\rca$), and the ones we consider are {\em weak K\"onig's lemma} ($\wkl$), {\em arithmetical comprehension} ($\aca$), and {\em $\Pi^1_1$-comprehension} ($\pica$), listed in strictly increasing order of strength.
Each of $\rca$, $\aca$ and $\pica$ include axioms asserting that sets of the form $ \{ n \in \Nb : \varphi(n)\}$ exist, where in the case of $\rca$, $\varphi$ must express a computable predicate; in the case of $\aca$, $\varphi$ may contain arbitrary quantifiers over natural numbers (but not over sets of natural numbers); and in the case of $\pica$, $\varphi$ is of the form $ \forall X \psi (n,X)$, where $X$ is a second-order variable and $\psi(n,X)$ contains no additional second-order quantifiers. Meanwhile, the theory $\wkl$ asserts that any infinite binary tree has an infinite branch.

It is known, for example, that the Baire category theorem is provable in $\rca$, that the Heine--Borel theorem is equivalent to $\wkl$, and that the Stone--Weierstrass theorem is equivalent to $\aca$.  As we will see throughout the text, there are many more examples of theorems equivalent to these theories.  On the other hand, statements equivalent to $\pica$ are much more difficult to come by.  One example is the Cantor--Bendixson theorem, stating that any closed subset of a Polish space can be written as the union of a countable set and a perfect set.

From a logical perspective, there is a vast gap in strength between the systems $\aca$ and $\pica$.  $\aca$ may be thought of as the second-order analog of the familiar first-order system of Peano arithmetic ($\pa$).  Indeed, $\aca$ is conservative over $\pa$, meaning that for every first-order statement $\varphi$ (i.e., for every $\varphi$ that may contain quantifiers over natural numbers but not over sets of natural numbers), $\varphi$ is provable in $\aca$ if and only if it is provable in $\pa$.  Both $\aca$ and $\pa$ have proof-theoretic ordinal $\varepsilon_0$.  In computational terms, $\aca$ is characterized by the existence of Turing jumps.  That is, over $\rca$, $\aca$ is equivalent to the statement ``for every set $X$, the Turing jump of $X$ is also a set.''  $\pica$, on the other hand, is a squarely impredicative system whose proof-theoretic ordinal is far above $\Gamma_0$.  In fact, its proof-theoretic ordinal is far above even the Bachmann--Howard ordinal.  In computational terms, $\pica$ is characterized by the existence of hyperjumps.  That is, over $\rca$, $\pica$ is equivalent to the statement ``for every set $X$, the hyperjump of $X$ is also a set.''  See~\cite{SimpsonSOSOA} for analyses of the Heine--Borel, Stone--Weierstrass, and Cantor--Bendixson theorems and for characterizations of $\aca$ and $\pica$ in terms of jumps and hyperjumps, and see~\cite{Pohlers1998} for an ordinal analysis of $\pica$.

Our goal in this article is to study Ekeland's variational principle~\cite{ekeland1974} in the context of reverse mathematics.  This principle states that under certain conditions, lower semi-continuous functions on complete metric spaces always attain `approximate minima,' which we call {\em critical points}.  In his original paper~\cite{ekeland1974} and the survey~\cite{ekeland1979}, Ekeland provides many applications of his variational principle, centered around optimization problems concerning minimal surfaces, partial differential equations, geodesics, the geometry of Banach spaces, control theory, and more.  Ekeland's variational principle has been studied extensively, leading to many variants and extensions (see e.g.~\cites{Georgiev1988,bosch2007}).  The variational principle can also be used to give an easy proof of Caristi's fixed point theorem~\cite{caristi1976}, whose logical strength we analyze in forthcoming work~\cites{LC17Talk, PLS2017Talk, ASL18Talk}.

{\em A priori}, an analysis of Ekeland's variational principle in reverse mathematics is a natural and interesting project.
Ekeland's variational principle is a well-known and important result, and understanding its computational content could lead to algorithms for approximating critical points, or at least to determining that no such algorithm exists.
From a technical perspective, lower semi-continuous functions have not yet received much attention in reverse mathematics, and developing a theory of these functions in second-order arithmetic is an interesting endeavor on its own (see \S\ref{secSemiCont}).

However, {\em a posteriori} the analysis of Ekeland's variational principle is even more interesting and quite surprising: as we will see, natural restrictions of the result (e.g.~to compact spaces, to continuous $f$) yield statements equivalent over $\rca$ to each of $\wkl$, $\aca$, and $\pica$, including what is, to the best of our knowledge, the first statement about continuous functions stemming from analysis that is equivalent to $\pica$.

Before diving into formal systems, let us discuss the variational principle in more detail and sketch a proof.
Let $\overline \Rb = \Rb \cup \{\pm \infty\}$, where we stipulate that $\sup \varnothing = -\infty$ and $\inf \varnothing = \infty$.
If $\spc X$ is a metric space, recall that a function $\fnc \colon \spc X \to \overline \Rb$ is {\em lower semi-continuous} if for every $x\in \spc X$ and every $\lambda < f(x)$, there is a $\delta > 0$ such that whenever $d(x,y) < \delta$, it follows that $\fnc (y) > \lambda$. The notion of an {\em upper semi-continuous} function is defined dually. Clearly, a function that is both upper and lower semi-continuous is continuous.  Additionally, recall the following well-known characterization of semi-continuity.

\begin{Lemma}\label{LemmSCF}
Let $\spc X$ be a metric space.
A function $\fnc \colon \spc X \imp \overline \Rb$ is
\begin{enumerate}[label=(\alph*)]
\item lower semi-continuous if and only if $\{(x, y) \in \spc X \times \Rb : \fnc(x) \leq y\}$ is closed;
\item upper semi-continuous if and only if $\{(x, y) \in \spc X \times \Rb : \fnc(x) \geq y\}$ is closed.
\end{enumerate}
\end{Lemma}

If $\fnc \colon \spc X \to \Rb_{\geq 0}$ is continuous but $\spc X$ is not compact, then $\fnc$ may not attain its infimum, and since every continuous function is also lower semi-continuous, the same holds of lower semi-continuous functions. Nevertheless, Ekeland's variational principle states that $\fnc$ has points that are in a sense approximate local minima, and we call these points {\em $\varepsilon$-critical points}. For the sake of brevity, we will often refer to lower semi-continuous functions $\fnc \colon \spc X \imp \overline{\Rb}_{\geq 0}$ as {\em potentials}.

\begin{Definition}\label{DefCritical}
Let $\spc X$ be a metric space, let $\varepsilon>0$, and let $\fnc \colon \spc X \imp \overline{\Rb}_{\geq 0}$ be a potential.  A point $x_* \in \spc X$ is an \emph{$\varepsilon$-critical point of $\fnc$} if
\begin{align*}
(\forall y \in \spc X)[( \varepsilon d(x_*, y)  \leq \fnc(x_*) -  \fnc(y)  ) \imp y = x_*].
\end{align*}
\end{Definition}

Ekeland's variational principle then reads as follows.

\begin{Theorem}[Ekeland \cite{ekeland1974}]\label{TheoEVP}
If $\spc X$ is a complete metric space and $\fnc \colon \spc X \imp \overline{\Rb}_{\geq 0}$ is any potential, then for every $\varepsilon > 0$, $\fnc$ has an $\varepsilon$-critical point $x_\ast$.

Moreover, for any $x_0 \in \spc X$, we can choose $x_\ast$ so that
\begin{equation}\label{localCondition}
\varepsilon d( x_ 0, x_\ast) \leq \fnc (x_0) - \fnc (x_\ast) .
\end{equation}
\end{Theorem}

In the literature, it is often assumed that $\fnc(x_0) \leq \varepsilon + \inf (f)$, in which case \eqref{localCondition} is replaced by $d(x_0, x_\ast) \leq 1$.
We refer to Theorem \ref{TheoEVP} without \eqref{localCondition} as the {\em free variational principle} ($\fvp$) and with \eqref{localCondition} as the {\em localized variational principle} ($\lvp$); note that Ekeland instead calls these the {\em weak} and {\em strong} principles, respectively.

Let us sketch a proof similar to that presented by Du~\cite{Du16}, which is based on one given by Br\'ezis and Browder~\cite{Brezis1976} in a more general order-theoretic setting.
Suppose that $\spc X$ is a complete metric space, let $x_0\in \spc X$, and let $\fnc \colon \spc X \imp \overline{\Rb}_{\geq 0}$ be any potential. For $\varepsilon > 0$, define a partial order $\peq_\varepsilon$ on $\spc X$ given by $y \peq_\varepsilon x$ if and only if $ \varepsilon d(x,y) \leq \fnc (x) -\fnc (y)$. For $x\in \spc X$, define $S(x) = \{y \in \spc X : y\peq_\varepsilon x\}$. Then we construct a sequence $(x_n : n<\omega)$ by letting $x_{n+1} \in S(x_n) $ be such that $f(x_{n+1}) < \inf f[ S(x_n) ] + 2^{-n}$.
It is not too hard to check that $(x_n : n<\omega)$ is Cauchy and hence converges to some $x_\ast \in \spc X$, and, using the fact that $\fnc$ is lower semi-continuous, that $\fnc(x_\ast) \leq \fnc(x_n)$ for all $n<\omega$. From this it readily follows that $x_\ast$ is $\varepsilon$-critical and satisfies \eqref{localCondition}; we leave the details to the reader.

As it turns out, there are some issues when attempting to formalize this argument in second-order arithmetic. The sets $S(x_n)$ are closed, and finding infima of lower semi-continuous functions on closed sets requires $\pica$ in general. It is easier to formalize the argument when $S(x_n)$ is open {\em and} $\fnc$ is continuous, and we can achieve this by replacing $S(x_n)$ by a suitable open set.  However, this modification typically causes \eqref{localCondition} to fail.
In fact, we obtain that the $\fvp$ for continuous functions is equivalent to $\aca$ (\S\ref{SecFormalEVP}), while the $\lvp$ is equivalent to $\pica$, even when $f$ is assumed to be continuous (\S\ref{secLVP}).  This last result is particularly interesting because it gives an example of a statement about continuous functions that is equivalent to $\pica$; typically the mathematics of continuous functions can be carried out in $\aca$ and below.

By tweaking other parameters in the statement, for example taking $\spc X$ to be compact or even to be a specific space such as the unit interval, we obtain several variants of the variational principle, each equivalent to one of $\wkl$, $\aca$ or $\pica$.
The equivalence between these theories and variational principles is established over $\rca$, and hinges on known equivalences, e.g.~the equivalence between $\aca$ and the fact that every increasing sequence of rationals bounded above has a supremum. The main reversals are found in \S\ref{secReversals}.
A full summary of our main results is found in \S\ref{secConc}.

\section{Subsystems of second-order arithmetic}\label{SecSSOA}

In this section we review the subsystems of second-order arithmetic we will be working with, loosely following the presentation of Simpson \cite{SimpsonSOSOA}.
Let us first settle conventions regarding syntax.  The language of second-order arithmetic consists of first-order variables intended to range over natural numbers; second-order variables intended to range over sets of natural numbers; constant symbols ${\tt 0}$ and ${\tt 1}$; $2$-ary function symbols $+$ and $\times$; and $2$-ary relation symbols $=$, $<$, and $\in$.  Terms and formulas are built from variables, constant symbols, function symbols, relation symbols, propositional connectives ($\neg$, $\wedge$, $\vee$, etc.), and the quantifiers $\forall$ and $\exists$.  The equality relation symbol is restricted to first-order objects.  Equality for second-order objects is defined via $\in$, and we write $X = Y$ as an abbreviation for $\forall n(n \in X \biimp n \in Y)$.  Similarly, we write $X \subseteq Y$ as an abbreviation for $\forall n(n \in X \imp n \in Y)$.

We use ${\Delta}^0_0$ to denote the class of all formulas, possibly with parameters, where no second-order quantifiers appear and all first-order quantifiers are {\em bounded;} that is, of the form $\forall x (x< t \rightarrow  \varphi)$ or $\exists x (x < t \wedge \varphi)$. We simultaneously define ${\Sigma}^0_{0}={\Pi}^0_{0}={\Delta}^0_0$, and we recursively define ${\Sigma}^0_{n+1}$ to be the set of all formulas of the form $\exists x \varphi$ with $\varphi\in {\Pi}^0_{n}$ and similarly define ${\Pi}^0_{n+1}$ to be the set of all formulas of the form $\forall x \varphi$ with $\varphi\in {\Sigma}^0_{n}$. We denote by ${\Pi}^0_\omega$ the union of all ${\Pi}^0_n$; these are the {\em arithmetical formulas.} 
The classes ${\Sigma}^1_n$ and ${\Pi}^1_n$ are defined analogously by counting alternations of second-order quantifiers and setting ${\Sigma}_0^1 = {\Pi}^1_0 = {\Delta}^1_0 = {\Pi}^0_\omega$.

We use Robinson arithmetic $\sf Q$ as our background theory, which essentially consists of the axioms of $\pa$ without induction (see~\cite[Definition~I.1.1]{HajekPudlak:1993:Metamathematics}).  When added to the basic axioms of $\sf Q$, the following schemes axiomatize many of the theories we consider.  Below, $\Gamma$ denotes a class of formulas.

\begin{itemize}

\item $\Gamma\mbox{-}\compax: \ \exists X\forall x\ \big (x\in X\leftrightarrow \varphi(x)\big )$,  where $\varphi\in\Gamma$ and $X$ is not free in $\varphi$;

\item ${\Delta}^0_1\mbox{-}\compax: \ \forall x \big (\pi(x)\leftrightarrow\sigma(x) \big )\rightarrow\exists X\forall x\ \big (x\in X\leftrightarrow \sigma(x)\big )$, where $\sigma\in{\Sigma}^0_1$, $\pi\in{\Pi}^0_1$, and $X$ is not free in $\sigma$;

\item ${\rm I}\Gamma : \ \varphi({\tt 0})\wedge\forall x\, \big (\varphi(x) \to\varphi(x+ {\tt 1}) \big )\ \to\ \forall x \ \varphi(x)$, where $\varphi\in\Gamma$.
\end{itemize}

Recursive comprehension, the base theory of second-order arithmetic, is defined as
\[\rca : \ {\sf Q} + {\rm I} \Sigma^0_1 + {\Delta}^0_1\mbox{-}\compax.\]
We also define the theory of arithmetical comprehension $\aca\colon \rca + \Sigma^ 0 _1$-$\compax$ and the much stronger theory $\pica \colon \aca + \Pi^ 1 _1$-$\compax$.  This gives us three of the four important theories that we consider; the fourth, $\wkl$, requires formalizing trees in second-order arithmetic.

$\rca$ suffices to define a bijective pairing function $\langle \cdot, \cdot \rangle \colon \Nb^{2} \to \Nb$ that is increasing in both coordinates, such as $\langle x,y\rangle = (x+y)(x+y+1)/2$.  We also use $\langle \cdot,\cdot \rangle$ to denote a pairing function on sets, which may also be defined in a standard way. With this, a binary relation is a pair of sets $\langle A,R\rangle $ with $A,R\subseteq \Nb$, where the elements of $R$ are of the form $\langle n,m\rangle$ with $n,m\in A$. When a binary relation $F$ is meant to represent a function, we write $y=F(x)$ instead of $\langle x,y\rangle\in F$.

$\rca$ also suffices to implement the typical codings of finite sequences of natural numbers as natural numbers (see~\cite[Section~II.2]{SimpsonSOSOA}, for example).  The set of all finite sequences of natural numbers is denoted $\BaireSet$. We write $\sigma\secseq \tau $ if $\sigma$ is an initial segment of $\tau$, $\sigma\secsub \tau $ if $\sigma$ is a proper initial segment of $\tau$, and define $\mathop\downarrow \sigma=\{\tau \in \BaireSet:\tau \secseq \sigma\}$.  If $x \colon \Nb \to \Nb$ and $n \in \Nb$, we write $ x \rest n$ for the finite sequence $(x (i))_{i<n}$. We extend the use of $\secsub$ by writing $\sigma \secsub x$ whenever $\sigma = x \rest n$ for some $n$.

\begin{Definition}
A {\em tree} is a set $T\subseteq \BaireSet$ such that $\forall \sigma(\sigma \in T \imp \da\sigma \subseteq T)$.  We say that $T$ is a {\em binary tree} if $T\subseteq \{0,1\}^{< \Nb}$, that is, if all entries appearing in elements of $T$ are either $0$ or $1$.  When $T\subseteq \BaireSet$ is a tree, we say that an infinite sequence $x$ is a \emph{path through $T$} if $\forall n (x \rest n\in T)$. The collection of all paths through $T$ is denoted $[T]$.
\end{Definition}

Say that a set $X\subseteq \Nb$ is {\em finite} if it is bounded above, and say that $X$ is infinite otherwise. With this, we define the axiom $\rm WKL$ to be the natural formalization of the following theorem.

\begin{Theorem}[K\"onig]
Every infinite binary tree has an infinite path. 
\end{Theorem}

For reference, we list the theories we have defined in increasing order of strength:

\begin{center}
\begin{tabular}{ll}
$\rca:$&${\sf Q} + {\rm I}{ \Sigma}^0_1 + { \Delta}^0_1$-$\compax$;\\
$\wkl :$&$\rca + \rm WKL$;\\
$\aca :$&$\rca + { \Sigma}^0_1$-$\compax$;\\
$\pica:$&$\aca + { \Pi}^1_1$-$\compax$.\\
\end{tabular}
\end{center}

The standard reference for second-order arithmetic is Simpson's~\cite{SimpsonSOSOA}, and we refer the reader to it for a complete treatment of all the material mentioned above.

\section{Metric spaces in second-order arithmetic}\label{sec-MetricDefs}

Part of the appeal of second-order arithmetic as a foundational system for mathematics is that it suffices to develop a large part of mathematical analysis, particularly when dealing with complete separable metric spaces. However, this requires some coding machinery. In this section, we review this machinery and establish notation that will be used throughout.

First, we assume that $\Qb$ is represented in some standard way using e.g.~pairs of natural numbers and that $\Rb$ is represented by rapidly converging Cauchy sequences of rationals as in~\cite[Section~II.4]{SimpsonSOSOA}.  We use the notation $\Qb_{>0}= \{q \in \Qb: q > 0 \}$, and we define $\Qb_{\geq 0}$, $\Rb_{\geq 0}$, etc.\ analogously.

\begin{Definition}[{$\rca$; see~\cite[Definition~II.5.1]{SimpsonSOSOA}}]
A (code for a) complete separable metric space $\spc X = \widehat X$ is a non-empty set $X \subseteq \Nb$ together with a sequence of real numbers $d\colon X \times X \to \Rb_{\geq 0}$ such that $d(a, a) = 0$, $d(a, b) = d(b, a) \geq 0$, and $d(a, b)+d(b, c) \geq d(a, c)$ for all $a, b, c \in X$. A {\em point} of $\widehat X$ is a sequence $x = (x_i)_{i \in \Nb}$ of elements of $X$ such that for all $i\leq j$, $d(x_i , x_j)\leq  2^{-i}$. We write $x\in \widehat X$ to mean that
$x$ is a point of $\widehat X$.
We identify $a\in X$ with the sequence $(a)_{i \in \Nb}$ and consider $X$ as a dense subset of $\widehat X$.
We set $d(x,y)=\lim_{n\to\infty}d(x_n,y_n)$, and write $x =_{\spc X} y$ if $d(x,y)=0$ (subscripts will be omitted if there is no confusion).
\end{Definition}

We use either notation $\spc X$ or $\widehat X$ to denote complete separable metric spaces.  The symbol $\widehat X$ emphasizes that the space is coded by the dense set $X$ (and the metric on $X$).  The following spaces appear throughout.

\begin{enumerate}

\item The real line, $\Rb$, with dense set the rational numbers, $\Qb$, equipped with the usual metric.
Closed subintervals of $\Rb$ may be represented similarly, where the dense set for $[a,b]$ is $\Qb \cap [a,b]$.

\item The Baire space, with dense set the set of sequences $x \in \Nb^\Nb$ that are eventually zero and
\[d(x,y) = \max \big ( \{ 0 \} \cup \big \{2^{-n}: x(n) \neq y(n) \} \big ). \] 

\item The Cantor space, which is $\{0,1\}^\Nb$ (also denoted $2^\Nb$) seen as a subspace of the Baire space.

\item $\mathcal C \big ( [a,b] \big )$ with $a<b$ rational numbers. The dense set is given by piecewise linear continuous functions $f\colon [a,b] \to \Rb$ with rational breakpoints, each represented by finitely many pairs $\langle x,f(x) \rangle \in \Qb \times \Qb$.  The metric is $d(f,g) = \max_{x\in [a,b]} |f(x) - g(x)|$.\footnote{We emphasize that $\mathcal C \big ( [a,b] \big )$ is the space of uniformly continuous functions $f\colon [a,b] \to \Rb$ that have moduli of uniform continuity in the sense of~\cite[Definition IV.2.1]{SimpsonSOSOA}.  The statement ``every continuous function $f\colon [a,b] \to \Rb$ has a modulus of uniform continuity'' is equivalent to $\wkl$ over $\rca$.}

\end{enumerate}

In the cases of the Baire space and the Cantor space, a sequence that is eventually zero may be represented via an appropriate finite initial segment and hence as a natural number.  If $\sigma$ is a finite sequence, then where convenient we identify $\sigma$ with $\sigma^\smf 0^\Nb$, which is the infinite sequence beginning with $\sigma$ and is then identically zero.  We remark also that the official dense set for $\mathcal C \big ( [a,b] \big )$ in~\cite{SimpsonSOSOA} is the set of polynomial functions $f\colon [a,b] \to \Qb$ with rational coefficients, but piecewise linear functions are more convenient for us.  The two presentations are equivalent over $\rca$ (see \cite[Example II.10.3 and Lemma IV.2.4]{SimpsonSOSOA}).

\begin{Definition}[{$\rca$; see~\cite[Definition~II.5.6]{SimpsonSOSOA}}]
Let $\widehat X$ be a complete separable metric space. A {\em (code for a) rational open ball $\ball ar$} is an ordered pair $\la a,r\ra$, with $a\in X$ and $r \in \Qb_{>0}$.  We define $\ball ar\ballsub \ball bq$ if $d(a,b) + r < q$.

A {\em (code for an) open set} $\MC U$ in $\widehat X$ is a set $U \subseteq \Nb \times X \times \Qb_{>0}$, where a point $x \in \widehat X$ is said to belong to $\MC U$ (abbreviated $x \in  \MC U$) if it satisfies the ${\Sigma}^0_1$ condition $\exists n \exists a\exists r (d(x, a) < r \wedge (n, a, r) \in U)$.  A {\em (code for a) closed set $\mathcal C$} is the same as a code for its complement $\widehat X\setminus \mathcal  C$, except we define $x\in \mathcal  C$ if $x\not\in \widehat X\setminus \mathcal  C$.
\end{Definition}

The intuition behind the above definition is that a code for an open set is an enumeration of rational open balls.  If $U \subseteq \Nb \times X \times \Qb_{>0}$ codes an open set, then the rational open ball $\ball{a}{r}$ appears in the enumeration if there is an $n$ with $(n,a,r) \in U$.

We remark that for $\spc X \subseteq \Rb$, we may represent basic open sets in the form $(a-r,a+r)$ rather than $\ball ar$, and we will often prefer this representation.  On occasion we will regard closed balls (or, more precisely, closures of balls) as metric spaces in their own right, in which case if $\ball ar$ is a ball in $\widehat X $, then $\overline {\ball ar} $ is the subspace with dense set $\ball ar \cap X$.

We may also reason about compactness within second-order arithmetic.

\begin{Definition}[$\rca$; {see~\cite[Definition~III.2.3]{SimpsonSOSOA}}]\label{def-compact}
A complete separable metric space $\spc X$ is \emph{compact} if there is a sequence of finite sequences $((x_{i,j} : j < n_i) : i \in \Nb)$ of points in $\spc X$ such that
\begin{align*}
(\forall z \in \spc X)(\forall i \in \Nb)(\exists j < n_i)(d(z, x_{i,j}) < 2^{-i}).
\end{align*}
We say that $\spc X$ is {\em Heine--Borel compact} if for every sequence $(\ball{a_k}{r_k} : k \in \Nb)$ of rational open balls such that $(\forall x \in \spc X)(\exists k \in \Nb)(x \in \ball{a_k}{r_k})$, there is an $N \in \Nb$ such that $(\forall x \in \spc X)(\exists k < N)(x \in \ball{a_k}{r_k})$.
\end{Definition}

What we call `compact' might be more properly called `uniformly (or effectively) totally bounded,' but when working with complete separable metric spaces in second-order arithmetic, the convention is to use `compact' for this notion, and we do not wish to deviate.  $\wkl$ is required to show that every compact (in the above sense) complete separable metric space is Heine--Borel compact.

\begin{Theorem}[{see~\cite[Theorem~IV.1.2 and Theorem~IV.1.5]{SimpsonSOSOA}}]\label{thm-WKLCompact}
The following are equivalent over $\rca$.
\begin{enumerate}
\item $\wkl$.
\item Every compact complete separable metric space is Heine--Borel compact.
\item The unit interval $[0,1]$ is Heine--Borel compact.
\end{enumerate}
\end{Theorem}

Sequential compactness is an even stronger notion.  $\aca$ is required to show that every compact space is sequentially compact.
\begin{Theorem}[{see~\cite[Theorem~III.2.2 and Theorem~III.2.7]{SimpsonSOSOA}}]\label{thm-SeqCompact}
The following are equivalent over $\rca$.
\begin{enumerate}
\item $\aca$.
\item Every infinite sequence of points in a compact complete separable metric space has a convergent subsequence.
\item Every infinite sequence of points in $[0,1]$ has a convergent subsequence.
\item Every increasing (or decreasing) sequence of points in $[0,1]$ converges.
\item\label{itRational} Every increasing (or decreasing) sequence of rational points in $[0,1]$ converges.
\end{enumerate}
\end{Theorem}

Note that the equivalence with \eqref{itRational} is not stated explicitly in \cite{SimpsonSOSOA}, but it follows from the proof of Theorem~III.2.2 given there.  We conclude this section by defining continuous functions between complete separable metric spaces.  The idea is to code a continuous function $f \colon \spc X \imp \spc Y$ by an enumeration of (codes for) pairs of open balls $\la \ball{a}{r}, \ball{b}{q} \ra$, where $\ball{a}{r} \subseteq \spc X$ and $\ball{b}{q} \subseteq \spc Y$.  If the pair $\la \ball{a}{r}, \ball{b}{q} \ra$ appears in the enumeration, then this means that $f$ maps $\ball{a}{r}$ into the closure of $\ball{b}{q}$.

\begin{Definition}[$\rca$; {\cite[Definition~II.6.1]{SimpsonSOSOA}}]\label{def-continuous}
Let $\spc X = \widehat X$ and $\spc Y = \widehat Y$ be complete separable metric spaces.  A \emph{continuous partial function} $f\colon \spc X \to \spc Y$ is coded by a set $\Phi \subseteq \Nb \times X \times \Qb_{>0} \times Y \times \Qb_{>0}$ that satisfies the properties below.  Let us write $\ball ar \fcode \Phi \ball bs$ for $\exists n \ \big (\sel n, a, r, b, s \ser \in \Phi \big )$.  Then, for all $a, a' \in X$, all $b,b'\in Y$, and all $r, r',s,s' \in \Qb_{>0}$, $\Phi$ must satisfy:
\begin{enumerate}[label=({\sc cf}\arabic*)]
\item\label{cfone} if $\ball{a}{r} \fcode \Phi \ball{b}{s}$ and $\ball{a}{r} \fcode \Phi \ball{b'}{s'}$, then $d(b,b') \leq s+s'$;

\item\label{cftwo} if $\ball{a}{r} \fcode \Phi \ball{b}{s}$ and $\ball{a'}{r'} \ballsub \ball{a}{r}$, then $\ball{a'}{r'}  \fcode \Phi \ball{b}{s}$;

\item\label{cfthree} if $\ball{a}{r} \fcode \Phi \ball{b}{s}$ and $\ball{b}{s} \ballsub \ball{b'}{s'}$, then $\ball{a}{r} \fcode \Phi \ball{b'}{s'}$.
\end{enumerate}

A point $x \in \spc X$ is in the domain of the function $f$ coded by $\Phi$ if, for every $\varepsilon > 0$, there are $\ball{a}{r}\fcode \Phi \ball{b}{s}$ such that $d(x, a) < r$ and $s < \varepsilon$. If $x \in  \dom(f)$, we define the value $f(x)$ to be the unique point $y \in \spc Y$ such that $d(y, b) \leq  s$ for all $\ball{a}{r} \fcode \Phi \ball{b}{s}$ with $d(x, a) < r$.
A {\em continuous function} is a continuous partial function $f \colon \spc X \to \spc Y$ with $\dom(f) = \spc X$.
In case $\spc Y=\Rb$, we often write $\ball{a}{r} \fcode \Phi (u,v)$ for $\ball{a}{r} \fcode \Phi \ball{b}{s}$ with $u=b-s$ and $v=b+s$.
\end{Definition}

To reason about the values of coded continuous functions, it often helps to think in the following way.  Suppose that $\Phi$ codes a continuous function $f\colon \spc X \to \spc Y$.  If $x \in \spc X$ and $y \in \spc Y$ are such that for every $\varepsilon > 0$ there are $\ball ar \fcode \Phi \ball bs$ with $x \in \ball ar$, $y \in \ball bs$, and $s < \varepsilon$, then $f(x) = y$.

The following lemmas are useful for constructing codes of open sets, closed sets, and continuous functions.

\begin{Lemma}[{\cite[Lemma~II.5.7]{SimpsonSOSOA}}]\label{lem-code-for-open}
For a given $\Sigma^{0}_{1}$ formula $\varphi(x)$, the following is provable within $\rca$.
Let $\spc X = \widehat X$ be a complete separable metric space.
If $x=_{\spc X}y$ implies $\varphi(x) \leftrightarrow \varphi(y)$, then there exists (a code for) an open set $\MC U \subseteq \spc X$ such that $\varphi(x)\leftrightarrow x\in \MC U$.
\end{Lemma}
This lemma guarantees that a $\Sigma^{0}_{1}$-definable subset of a complete separable metric space is an open set, and thus a $\Pi^{0}_{1}$-definable subset is a closed set.
Note that this lemma holds uniformly. In other words, if $\varphi(x,i)$ is a $\Sigma^{0}_{1}$ formula that defines a subset of $\spc X$ for each $i\in\Nb$, then there exists a sequence of codes for open sets $(\MC{U}_{i})_{i\in\Nb}$ such that $\varphi(x,i)\leftrightarrow x\in \MC{U}_{i}$.

Intuitively, the following lemma states that if a function is uniformly continuous in an effective way and the values of the function can be computed on a dense set (e.g., if $f(x)$ is provided by elementary functions, by power series, etc.), then there is a code for the function.  This property also holds restricted to any open set.

\begin{Lemma}[$\rca$]\label{lem-code-for-conti1}
Let $\spc X = \widehat X$ and $\spc Y = \widehat Y$ be complete separable metric spaces, and let $\MC{U}\subseteq \spc X$ be an open set.
Assume that $(\la a_{i},y_{i} \ra : i \in \Nb)$ is a sequence of points in $X\times\spc Y$ and that $h\colon\Nb\to\Nb$ is a function such that
\begin{itemize}
 \item $(a_{i})_{i\in\Nb}$ enumerates all points in $\MC{U}\cap X$,
 \item $d_{\spc X}(a_{i},a_{j})<2^{-h(n)}$ implies $d_{\spc Y}(y_{i},y_{j})<2^{-n}$ for all $i,j,n\in\Nb$.
\end{itemize}
Then, there exists (a code for) a continuous function $f:\MC{U}\to \spc Y$ such that $f(a_{i})=y_{i}$ for all $i\in \Nb$.  (In fact, $h$ is a modulus of uniform continuity for $f$; see \cite[Definition~IV.2.1]{SimpsonSOSOA}.)
\end{Lemma}

\begin{proof}
Let $U$ be the code for $\MC{U}$.
Define a code $\Phi$ for a continuous partial function $f \colon \spc X \imp \Rb$ so that $\ball ar \fcode \Phi \ball bs$ if and only if  the quadruple $(a,r,b,s)\in X \times \Qb_{>0} \times Y \times \Qb_{>0}$ satisfies
\begin{itemize}
 \item $\ball{a}{r} \ballsub \ball{a'}{r'}$ for some $(n,a',r')\in U$, and
 \item $\ball ar \ballsub \ball{a_{i}}{2^{-h(n)-1}}$ and $\ball{y_{i}}{2^{-n}} \ballsub \ball{b}{s}$ for some $n,i\in\Nb$.
\end{itemize}
The code $\Phi$ exists because the above conditions can be described by a $\Sigma^{0}_{1}$ formula.
It is easy to see that $\Phi$ satisfies \ref{cftwo} and \ref{cfthree} of Definition~\ref{def-continuous}.
We show that $\Phi$ also satisfies \ref{cfone}.
Assume that $\ball ar \fcode \Phi \ball bs$ and $\ball ar \fcode \Phi \ball{b'}{s'}$, where in each case the second condition is satisfied by $n,a_{i},y_{i}$ and $m,a_{j},y_{j}$, respectively.
Without loss of generality, we may assume that $n\le m$ and that $h$ is non-decreasing.
Since $a\in \ball{a_{i}}{2^{-h(n)-1}}\cap \ball{a_{j}}{2^{-h(m)-1}}$, we have $d_{\MC{X}}(a_{i},a_{j})<2^{-h(n)}$.
Then $d_{\MC{Y}}(y_{i},y_{j})<2^{-n}$, and thus $y_{j}\in \ball{y_{i}}{2^{-n}}\subseteq \ball{b}{s}$.
Hence $y_{j}\in \ball{b}{s}\cap\ball{b'}{s'}$, which implies that $d_{\MC{Y}}(b,b')\leq s+s'$.

Finally, we check that $\MC{U}=\dom(f)$ and $f(a_{i})=y_{i}$.
If $x\notin \MC{U}$, then $x\notin\dom(f)$ by the first condition, thus we have $\dom(f)\subseteq\MC{U}$.
For the converse, assume that $x\in\MC{U}$ and $\varepsilon>0$.
Take $(n',a',r')\in U$ so that $x\in \ball{a'}{r'}$ and $n\in\Nb$ so that $2^{-n+1}<\varepsilon$, and then take $r\in \Qb_{>0}$ so that $\ball{x}{2r}\ballsub \ball{a'}{r'}$ and $r<2^{-h(n)-1}$.
Pick a point $a=a_{i}\in \MC{U}\cap X$ so that $d_{\MC{X}}(x,a)<r$ and a point $b\in Y$ so that $d_{\MC{Y}}(y_{i},b)<2^{-n}$, and put $s=2^{-n+1}$. Then we have $\ball{a}{r} \fcode \Phi \ball{b}{s}$ and $x\in \ball{a}{r}$. Since $\varepsilon$ was arbitrary and $s<\varepsilon$, this means that $x\in\dom(f)$.
If $x=a_{j}$ for some $j\in\Nb$ in the above, then $i$ can be always chosen to be equal to $j$, and thus one can find $\ball{a}{r} \fcode \Phi \ball{b}{s}$ with $a_{j}\in\ball ar$, $y_{j}\in \ball bs$ and $s<\varepsilon$ for all $\varepsilon>0$. This means that $f(a_{j})=y_{j}$ for all $j\in\Nb$.
\end{proof}

We also remark that continuous partial functions can be patched together, provided they are mutually compatible.

\begin{Lemma}[$\rca$; {\cite[Lemma~3.31]{Y-TMP}}]\label{lem-code-for-conti2}
Let $\spc X = \widehat X$ and $\spc Y = \widehat Y$ be complete separable metric spaces.
Assume that $(\MC{U}_{i},f_{i})_{i\in \Nb}$ is a sequence of pairs of (codes for) open sets $\MC{U}_{i}\subseteq\spc X$ and continuous functions $f_{i}:\MC{U}_{i}\to \spc Y$ such that $f_{i}(x)=f_{j}(x)$ for any $i,j\in \Nb$ and $x\in \MC{U}_{i}\cap \MC{U}_{j}$.
Then, there exists (a code for) a continuous function $\bar f:\bigcup_{i\in\Nb}\MC{U}_{i}\to \spc Y$ such that $\bar f(x)=f_{i}(x)$ for all $i\in \Nb$ and $x\in \MC{U}_{i}$.
\end{Lemma}

Notice that the above lemmas also hold uniformly.  For Lemma~\ref{lem-code-for-conti1}, if we are given a sequence $( \la a_{i,j},y_{i,j} \ra_{i\in \Nb}, \MC{U}_{j}, h_{j} : j\in\Nb)$, where $\la a_{i,j},y_{i,j} \ra_{i\in \Nb}$, $ \MC{U}_{j}$ and $h_{j}$ satisfy the conditions of Lemma~\ref{lem-code-for-conti1} for all $j\in\Nb$, then we may obtain a sequence of continuous functions $(f_{j}:\MC{U}_{j}\to \spc Y)_{j\in\Nb}$.

Finally, we consider products of finitely many metric spaces $(\spc{X}_i : i < n)$. In \cite[Example~II.5.4]{SimpsonSOSOA}, product spaces are defined via the Euclidean metric, but for our purposes it is more convenient to use the `max' metric instead.

\begin{Definition}[$\rca$; {see \cite[Example~II.5.4]{SimpsonSOSOA}}]\label{def-ProdSpace}
Let $(\spc{X}_i : i < n)$ be complete separable metric spaces coded by dense sets $(X_i : i < n)$ and metrics $(d_i : i < n)$.  The \emph{product space} $\spc X = \prod_{i<n}\spc{X}_i$ is the complete separable metric space coded by the dense set $X = \prod_{i<n}X_i$ and the metric $d \colon X \times X \imp \Rb$ where
\begin{align*}
d(\vec a, \vec b) = \max_{i<n}d_i(a_i,b_i).
\end{align*}
\end{Definition}

Products of metric spaces will be very useful to us.
For example, we may use them to define Banach spaces in the context of second-order arithmetic: these are simply metric spaces $\spc X$ equipped with a designated element ${\bm 0} \in \spc X$ and total, continuous functions ${+} \colon \spc X \times \spc X \to \spc X$ and ${\cdot} \colon \Rb \times \spc X \to \spc X$ such that if we define $\|x\| = d(x,{\bm 0})$, then $(\spc X,{\bm 0},{+},{\cdot},\|\cdot \|)$ satisfies the standard definition of a normed vector space.
On occasion we will make free use of the fact that some of the spaces we consider, such as ${\mathcal C}\big ( [0,1]\big )$, come equipped with a standard Banach space structure.

If $\spc X = \prod_{i<n}\spc{X}_i$ is the product of the complete separable metric spaces $(\spc{X}_i : i < n)$, then, officially, the basic open sets in $\spc X$ are those of the form $\ball{\vec a}{r}$, where $\vec a \in \prod_{i<n}X_i$ and $r \in \Qb_{>0}$.  However, when working with $\spc X$, it is often more convenient to think of the basic open sets as being the sets of the form $\prod_{i<n}\ball{a_i}{r_i}$, where each $\ball{a_i}{r_i}$ is a basic open set in $\spc{X}_i$.  The following lemma says that in $\rca$ we can uniformly translate between the two styles of basic open sets when coding open sets.  Therefore we may always use the style of open set that is most convenient.

\begin{Lemma}[$\rca$]
Let $(\spc{X}_i : i < n)$ be complete separable metric spaces coded by dense sets $(X_i : i < n)$ and metrics $(d_i : i < n)$.  Let $X = \prod_{i<n}X_i$ and $d$ code $\spc X = \prod_{i<n}\spc{X}_i$ as in Definition~\ref{def-ProdSpace}.  
\begin{itemize}
\item[(i)] There is a function $f \colon \prod_{i<n}(X_i \times \Qb_{>0}) \times \Nb \imp X \times \Qb_{>0}$ such that, for every
\[\sel a_0, r_0, \dots, a_{n-1}, r_{n-1} \ser \in \prod_{i<n} (X_i \times \Qb_{>0}),\]
$f(a_0, r_0, \dots, a_{n-1}, r_{n-1}, \cdot)$ enumerates an official code for the unofficial basic open set
\[\prod_{i<n}\ball{a_i}{r_i} \subseteq \spc X.\]

\item[(ii)] For every $\la \vec a, r \ra \in X \times \Qb_{>0}$, $\ball{\vec a}{r} = \prod_{i < n}\ball{a_i}{r}$.  Thus the function $g \colon X \times \Qb_{>0} \imp \prod_{i<n} (X_i \times \Qb_{>0})$ given by $g(\vec a, r) = \sel a_0, r, \dots, a_{n-1}, r \ser$ translates the code $\la \vec a, r \ra$ for the official basic open set $\ball{\vec a}{r}$ to the code $\sel a_0, r, \dots, a_{n-1}, r \ser$ for the equivalent unofficial basic open set $\prod_{i < n}\ball{a_i}{r}$.
\end{itemize}
\end{Lemma}

\begin{proof}
(i) The function $f$ exists by Lemma~\ref{lem-code-for-open} because the unofficial basic open set $\prod_{i<n}\ball{a_i}{r_i}$ is $\Sigma^0_1$ uniformly in $(a_0, r_0, \dots, a_{n-1}, r_{n-1})$.

(ii) Let $\vec x \in \spc X$.  Then $ \vec x \in \ball{\vec a}{r}$ if and only if $d(\vec a, \vec x) < r$ if and only if $\max_{i < n}d_i(a_i, x_i) < r$, if and only if $(\forall i < n)(d_i(a_i, x_i) < r)$, if and only if $\vec x \in \prod_{i < n}\ball{a_i}{r}$.  Thus $\ball{\vec a}{r} = \prod_{i < n}\ball{a_i}{r}$.
\end{proof}

\section{Semi-continuous functions in second-order arithmetic}\label{secSemiCont}

Although continuous functions have been extensively studied in the context of second-order arithmetic, lower semi-continuous functions have received less attention.
Fortunately, they admit a natural representation in the spirit of Definition \ref{def-continuous}.

\begin{Definition}[$\rca$]\label{def-lsc}
Let $\spc X = \widehat X$ be a complete separable metric space.  A \emph{code for a lower semi-continuous function} $\fnc\colon \spc X \to \overline \Rb$ is a set ${\Psi} \subseteq \Nb \times X \times \Qb_{>0} \times \Qb$ that satisfies the properties below.  Let us write $\ball{a}{r} \vcode \Psi q$ for $\exists n( ( n, a, r, q ) \in {\Psi})$.  Then, for all $a, a' \in X$, all $q, q' \in \Qb$, and all $r, r' \in \Qb_{>0}$, $\Psi$ must satisfy:

\begin{enumerate}[label=({\sc lsc}\arabic*)]

\item if $\ball{a}{r} \vcode \Psi q$ and $\ball{a'}{r'}  \ballsub \ball{a}{r}$, then $\ball{a'}{r'} \vcode \Psi q$, and

\item if $\ball{a}{r} \vcode \Psi q$ and $q' < q$, then $\ball{a}{r} \vcode \Psi q'$.

\end{enumerate}

For $x\in \spc X$ and $y\in \overline \Rb$ we define $\fnc (x) = y$ if
\begin{align*}
y = \sup\{q \in \Qb : (\exists \la a,r \ra \in X \times \Qb_{>0})(\ball{a}{r} \vcode \Psi q \andd d(x,a) < r)\}.
\end{align*}
The relation $f$ thus defined is a {\em lower semi-continuous function} if $(\forall x \in \spc X)(\exists y \in \overline\Rb)(\fnc(x) = y)$.
The {\em support} of $f$, denoted $\lscsupp(f)$, is the collection $\{x\in \spc X : f(x) \in \Rb \}$.  We do not assume that $\lscsupp(f)$ exists as any kind of coded set in $\rca$.  An assertion of the form $x \in \lscsupp(f)$ is to be taken as an abbreviation for $(\exists y \in \Rb)(f(x) = y)$.

If $b \in \Qb$ is such that for every $a \in X$ and $r \in  \Qb_{>0}$ there is $m\in \Nb$ such that $( m, a, r, b ) \in \Psi$, we say that $\Psi$ is a code of a lower semi-continuous function $\fnc \colon \spc X \to \overline{\Rb}_{\geq b}$.
If $b = 0$, we call the function coded by $\Psi$ a {\em potential}.
\end{Definition}

Suppose that $\spc X$ is a complete separable metric space.  As with Definition~\ref{def-continuous}, the idea behind Definition~\ref{def-lsc} is that $\Psi$ enumerates pairs $\la B_r(a), q \ra$ with the property that if $f$ is the function being coded by $\Psi$, then $f$ maps $\ball{a}{r} \cap \dom(f)$ into $[q, \infty]$.  One may define \emph{upper semi-continuous partial functions} from $\spc X$ to $\overline \Rb$ by appropriately dualizing Definition~\ref{def-lsc}, in which case we write $\ball ar \ucode \Psi q$ for the dual of $\ball ar \vcode \Psi q$.
We remark that according to our definition, any code for a lower semi-continuous function defines a function $f \colon \spc X \to \overline \Rb$, as, recalling that $\sup \varnothing = -\infty$, we see that every subset of $\Rb$ has a supremum in $\overline \Rb$.
However, this fact is not provable in $\rca$.
For this reason, lower semi-continuous functions are defined with the explicit assumption that $\fnc$ is defined everywhere.
See Remark~\ref{rem-LSC} for further discussion.

Next we show that $\rca$ proves a version of Lemma~\ref{LemmSCF}, which we take as evidence that Definition~\ref{def-lsc} is a reasonable definition of semi-continuity for use in $\rca$.
Indeed, we prove an $\rca$ version of the fact that a function $f \colon \spc X \imp \overline \Rb$ is lower semi-continuous if and only if $\{ \la x, y \ra \in \spc X \times \Rb : f(x) \leq y\}$ is closed. 

\begin{Proposition}[$\rca$]\label{prop-LSCandGraph}
Let $\spc X$ be a complete separable metric space.
\begin{enumerate}[label=(\roman*)]
\item \label{itLSCandGraphOne}  If $f \colon \spc X \imp \overline \Rb$ is lower semi-continuous, then $\{\la x, y \ra \in \spc X \times \Rb : f(x) \leq y\}$ is closed.

\item\label{itLSCandGraphTwo}  If $\MC C \subseteq \spc X \times \Rb$ is a closed set such that
	\begin{itemize}
	\item for every $x \in \spc X$ and $y,z \in \Rb$, if $\la x, y \ra \in \MC C$ and $z \geq y$, then $\la x, z \ra \in \MC C$, and
	\item for every $x \in \spc X$, $\inf\{y \in \Rb : \la x, y \ra \in \MC C\}$ exists
	\end{itemize}
then there is a lower semi-continuous function $f \colon \spc X \imp \overline \Rb $ such that $\MC C = \{\la x, y \ra \in \spc X \times \Rb : f(x) \leq y\}$.
\end{enumerate}
\end{Proposition}

\begin{proof}
\noindent \ref{itLSCandGraphOne}
Let $f$ be coded by $\Phi$.
Then $f(x)>y$ if and only if there exists $(a,r,q)\in X\times\Qb_{>0}\times\Qb$ such that $\ball{a}{r} \vcode \Phi q \andd d(x,a) < r\andd y<q$. Thus, $\{\la x, y \ra \in \spc X \times \Rb : f(x) > y\}$ is $\Sigma^{0}_{1}$-definable, hence it is open by Lemma~\ref{lem-code-for-open}.

\noindent \ref{itLSCandGraphTwo} Let $\MC C$ be a closed set as in the statement of \ref{itLSCandGraphTwo}, and let $\MC U$ be the complement of $\MC C$.  Enumerate a code $\Phi$ for a lower semi-continuous function $f$ by enumerating $\ball{a}{r} \vcode \Phi q$ whenever $\MC U$ enumerates an open set $\ball{b}{s} \times (u, v)$ with $\ball{a}{r} \ballsub \ball{b}{s}$ and $q < v$.  We show that, for every $x \in \spc X$, $f(x)$ exists and equals $\inf\{z \in \Rb : \la x, z \ra \in \MC C\}$ (which exists by the assumptions on $\MC C$).  So let $x \in \spc X$, and let $y = \inf\{z \in \Rb : \la x, z \ra \in \MC C\}$.  We need to show that
\begin{align*}
y = \sup\{q \in \Qb : (\exists \la a,r \ra \in X \times \Qb_{>0})(\ball{a}{r} \vcode \Phi q \andd d(x,a) < r)\}.
\end{align*}

Suppose that $q \in \Qb$ is such that there is a ball $\ball{a}{r} \subseteq \spc X$ with $x \in \ball{a}{r}$ and $\ball{a}{r} \vcode \Phi q$.  By the definition of $\Phi$, there is a $\ball{b}{s} \times (u, v) \subseteq \MC U$ where $\ball{a}{r} \ballsub \ball{b}{s}$ and $q < v$.  From this we can conclude that if $z \in \Rb$ is such that $\la x, z \ra \in \MC C$, then $q < z$.  If not, then by the first assumption on $\MC C$, there would be a $z$ such that $\la x, z \ra \in \MC C$ and $z \in (u,v)$.  This implies that $\la x, z \ra \in \ball{b}{s} \times (u,v) \subseteq \MC U$, which contradicts that $\MC U$ is the complement of $\MC C$.  Thus $q \leq y$ because $y = \inf\{z \in \Rb : \la x, z \ra \in \MC C\}$.

We have shown that $y$ is an upper bound on the $q \in \Qb$ for which there is a $\ball{a}{r} \subseteq \spc X$ with $x \in \ball{a}{r}$ and $\ball{a}{r} \vcode \Phi q$.  We need to show that $y$ is least.  So consider a $z < y$.  Then $\la x, z \ra \notin \MC C$ because $y = \inf\{z \in \Rb : \la x, z \ra \in \MC C\}$.  Then there is a $\ball{a}{s} \times (u,v)$ enumerated into $\MC U$ that contains $\la x, z \ra$.  Let $r < s$ be such that $x \in \ball{a}{r} \ballsub \ball{a}{s}$, and let $q \in \Qb \cap (z,v)$.  Then $\ball{a}{r} \vcode \Phi q$ by the definition of $\Phi$.  Thus $x \in \ball{a}{r}$ and $\ball{a}{r} \vcode \Phi q$, but $z < q$.  So $z$ is not an upper bound on the $q \in \Qb$ for which there is a $\ball{a}{r} \subseteq \spc X$ with $x \in \ball{a}{r}$ and $\ball{a}{r} \vcode \Phi q$.  Thus $y$ is indeed the required supremum.

We have shown that, for all $x \in \spc X$, $f(x) = \inf\{z \in \Rb : \la x, z \ra \in \MC C\}$.  It is now easy to see that, for $\la x, y \ra \in \spc X \times \Rb$, $f(x) \leq y$ if and only if $\inf\{z \in \Rb : \la x, z \ra \in \MC C\} \leq y$ if and only if $\la x, y \ra \in \MC C$.
\end{proof}

The analogous properties of upper semi-continuous functions can be proved by dualizing Proposition~\ref{prop-LSCandGraph}.

The next proposition is an $\rca$ version of the fact that a function $f \colon \spc X \imp \Rb$ is continuous if and only if it is upper semi-continuous and lower semi-continuous.

\begin{Proposition}[$\rca$]\label{prop-LSCchar}
Let $\spc X$ be a complete separable metric space.
\begin{enumerate}[label=(\roman*)]
\item \label{itContisLSC} If $f \colon \spc X \imp \Rb$ is continuous, then there are a lower semi-continuous $g \colon \spc X \imp \Rb$ and an upper semi-continuous $h \colon \spc X \imp \Rb$ such that $(\forall x \in \spc X)(f(x) = g(x) = h(x))$.

\item \label{itLSCisCont}  If $g \colon \spc X \imp \overline \Rb$ is lower semi-continuous, $h \colon \spc X \imp \overline \Rb$ is upper semi-continuous, and $(\forall x \in \spc X)(g(x) = h(x))$, then there is a continuous partial $f \colon \spc X \imp \Rb$ with $\dom(f) = \lscsupp(g) = \lscsupp(h)$ such that $(\forall x \in \dom(f))(f(x) = g(x) = h(x))$.
\end{enumerate}
\end{Proposition}

\begin{proof}
\noindent \ref{itContisLSC} Let $\Phi$ be a code for a continuous $f \colon \spc X \imp \Rb$.  Define a code $\Gamma$ for a lower semi-continuous $g \colon \spc X \imp \overline \Rb$ by enumerating $\ball{a}{r} \vcode \Gamma q$ whenever $\ball{b}{s} \fcode \Phi (u,v)$ for a $\ball{b}{s}$ with $\ball{a}{r} \ballsub \ball{b}{s}$ and a $(u,v)$ with $q < u$.  It is then easy to check that $g(x)$ is defined and equal to $f(x)$ for all $x \in \spc X$.  The upper semi-continuous partial function $h \colon \spc X \imp \overline \Rb$ such that $(\forall x \in \dom (f)) (f(x) = h(x))$ is defined dually.\\

\noindent \ref{itLSCisCont} Let $\Gamma$ be a code for a lower semi-continuous $g \colon \spc X \imp \overline \Rb$ and let $\Psi$ be a code for an upper semi-continuous $h \colon \spc X \imp \overline \Rb$ such that $(\forall x \in \spc X)(g(x) = h(x))$.  Define a code $\Phi$ for a continuous partial function $f \colon \spc X \imp \Rb$ by enumerating $\ball{a}{r} \fcode \Phi (u, v)$ whenever there are $\ball{b}{s}, \ball{c}{t} \subseteq \spc X$ and $p, q \in \Qb$ such that $\ball{a}{r} \ballsub \ball{b}{s}$, $\ball{a}{r} \ballsub \ball{c}{t}$, $[p,q] \subseteq (u,v)$, $\ball{b}{s} \vcode \Gamma p$, and $\ball{c}{t} \ucode \Psi q$.  It is easy to see that $\Phi$ satisfies \ref{cftwo} and \ref{cfthree} of Definition~\ref{def-continuous}.  To see that $\Phi$ satisfies \ref{cfone} of Definition~\ref{def-continuous}, first observe that if $\ball{a}{r} \fcode \Phi (u,v)$, then $g(a) > u$ and $h(a) < v$.  However, $g(a) = h(a)$, so their common value must be finite and in the interval $(u, v)$.  Therefore, if $\ball{a}{r} \fcode \Phi (u, v)$ and $\ball{a}{r} \fcode \Phi (u', v')$, then $(u, v)$ and $(u', v')$ both contain $g(a)$, and therefore $(u, v)$ and $(u', v')$ intersect.

We show that $\dom(f) = \lscsupp(g) = \lscsupp(h)$ and that $(\forall x \in \dom(f))(f(x) = g(x) = h(x))$.  First, suppose that $x \notin \lscsupp(g)$.  Then $g(x) = \pm \infty$.  Suppose that $g(x) = -\infty$.  The only way this can happen is if there is no $\ball{b}{s} \vcode \Gamma p$ with $x \in \ball{b}{s}$.  But if there is no $\ball{b}{s} \vcode \Gamma p$ with $x \in \ball{b}{s}$, then there is also no $\ball{a}{r} \fcode \Phi (u, v)$ with $x \in \ball{a}{r}$.  Thus $x \notin \dom(f)$.  If instead $g(x) = \infty$, then also $h(x) = \infty$.  In this case, there can be no $\ball{c}{t} \ucode \Psi q$ with $x \in \ball{c}{t}$, which similarly implies that $x \notin \dom(f)$.

Now suppose that $x \in \lscsupp(g)$, and let $\varepsilon \in \Qb_{>0}$.  Let $(u, v) \subseteq \Rb$ be such that $g(x) \in (u, v)$ and $v - u < \varepsilon$.  By the definition of $g(x)$, there are a $\ball{b}{s} \subseteq \spc X$ and a $p \in (u, v)$ such that $x \in \ball{b}{s}$ and $\ball{b}{s} \vcode \Gamma p$.  Likewise, as $h(x) = g(x) \in (u,v)$, there are a $\ball{c}{t} \subseteq \spc X$ and a $q \in (u, v)$ such that $x \in \ball{c}{t}$ and $\ball{c}{t} \ucode \Psi q$.  Notice then that $p \leq g(x) \leq q$.  Let $\ball{a}{r} \subseteq \spc X$ be such that $x \in \ball{a}{r}$, $\ball{a}{r} \ballsub \ball{b}{s}$, and $\ball{a}{r} \ballsub \ball{c}{t}$.  Then $\ball{b}{s}$, $\ball{c}{t}$, $p$, and $q$ witness that $\ball{a}{r} \fcode \Phi (u, v)$.  Thus for every $\varepsilon \in \Qb_{>0}$, there are a $\ball{a}{r} \subseteq \spc X$ with $x \in \ball{a}{r}$ and a $(u,v) \subseteq \Rb$ with $v-u < \varepsilon$ and $g(x) \in (u,v)$.  This means that $x \in \dom(f)$ and $f(x) = g(x) = h(x)$.
\end{proof}

\begin{remark}\label{rem-LSC}
Our definition of the values of lower semi-continuous functions by suprema of rationals is similar to the definition of Borel measures within $\rca$ in~\cite[Section~X.1]{SimpsonSOSOA}.  As with Borel measures, one could understand the values of lower semi-continuous functions in a comparative way instead of requiring that the defining suprema exist.  That is, if $f$ is lower semi-continuous, then it is still possible to make sense of inequalities like $f(x) \leq r$ in $\rca$ even if the supremum defining $f(x)$ does not exist.  Indeed, in $\omega$-models, the values of lower semi-continuous functions are defined as (relative) left-c.e.\ reals.  With this perspective, Proposition~\ref{prop-LSCandGraph} is still true without the existence of infima.  The proof of Proposition~\ref{prop-LSCchar} shows that a code $\Psi$ for a partial continuous function $g\colon \spc X\to \Rb$ can be considered as a pair of codes $(\Psi_{-},\Psi_{+})$ for lower and upper semi-continuous functions by putting $\ball{a}{r} \vcode {\Psi_{-}}  u$ and $\ball{a}{r} \ucode {\Psi_{+}} v$ if $\ball{a}{r} \fcode \Psi (u,v)$, and then $x\in \dom(g)$ if and only if the values at $x$ defined by $\Psi_{-}$ and $\Psi_{+}$ coincide.  If the values coincide in a comparative sense, then it exists as a real number within $\rca$.  Similar modifications may be available for other theorems presented here.
\end{remark}

\section{Honestly-coded potentials}

Recall that we refer to lower semi-continuous functions $f \colon \spc X \to \overline{\Rb}_{\geq 0}$ as {\em potentials}.  According to Definition~\ref{def-lsc}, for $\spc X$ a complete separable metric space, $\Phi$ a code for a lower semi-continuous function $f \colon \spc X \imp \overline \Rb$, $\ball{a}{r} \subseteq \spc X$, and $q \in \Qb$, we know that $f(x) \geq q$ for all $x \in \ball{a}{r}$ if it is the case that $\ball{a}{r} \vcode \Phi q$.  We call a code $\Phi$ for $f$ \emph{honest} if it contains all of the information of this sort.  That is, if the `if' is an `if and only if.'

\begin{Definition}[$\rca$]\label{def-honest}
Let $\spc X$ be a complete separable metric space and let $f \colon \spc X \imp \overline \Rb$ be lower semi-continuous.  A code $\Phi$ for $f$ is called \emph{honest} if, for every $\ball{a}{r} \subseteq \spc X$ and every $q \in \Qb$, $\ball{a}{r} \vcode \Phi q$ if and only if $(\forall x \in \ball{a}{r})(f(x) \geq q)$.  If $f$ has an honest code, then we say that $f$ is \emph{honestly-coded}.
\end{Definition}

Notice that if $f \colon \spc X \imp \overline \Rb$ is lower semi-continuous and $\Phi \subseteq \Nb \times X \times \Qb_{>0} \times \Qb$ is any set such that, for every $\ball{a}{r} \subseteq \spc X$ and every $q \in \Qb$, $\ball{a}{r} \vcode \Phi q$ if and only if $(\forall x \in \ball{a}{r})(f(x) \geq q)$, then $\Phi$ is automatically a code (and hence an honest code) for $f$ as in Definition~\ref{def-lsc}.

Every lower semi-continuous function admits an honest code, although such a code cannot always be constructed in a weak theory.

\begin{Lemma}\label{lem-HonestCodes}
Let $\spc X$ be a complete separable metric space.
\begin{enumerate}[label=(\roman*)]
\item\label{itHonestOne} \textup{(}$\wkl$\textup{)} If $\spc X$ is compact and $f \colon \spc X \imp \overline \Rb$ is lower semi-continuous, then there is a code $\Phi$ for $f$ such that, for every $\ball{a}{r} \subseteq \spc X$ and every $q \in \Qb$, $\ball{a}{r} \vcode \Phi q$ if and only if $(\forall x \in \overline{\ball{a}{r}})(f(x) > q)$.

\item\label{itHonestTwo} \textup{(}$\aca$\textup{)} If $\spc X$ is compact and $f \colon \spc X \imp \overline \Rb$ is lower semi-continuous, then $f$ has an honest code.

\item\label{itHonestThree} \textup{(}$\aca$\textup{)} If $f \colon \spc X \imp \Rb$ is continuous, then there is an honestly-coded lower semi-continuous $g \colon \spc X \imp \overline \Rb$ such that $(\forall x \in \spc X )(g(x) = f(x))$.
 
\item\label{itHonestFour} \textup{(}$\pica$\textup{)} If $f \colon \spc X \imp \overline \Rb$ is lower semi-continuous, then $f$ has an honest code.
\end{enumerate}
\end{Lemma}

\begin{proof}
\ref{itHonestOne} Work in $\wkl$.  For $\la a, r \ra \in X \times \Qb$ and a sequence $( \la b_i, s_i \ra : i < n )$ of elements of $X \times \Qb$, the assertion that $\overline{\ball{a}{r}} \subseteq \bigcup_{i < n}\ball{b_i}{s_i}$ is $\Sigma^0_1$ uniformly in $\la a, r \ra$ and $( \la b_i, s_i \ra : i < n )$, essentially by a uniform version of~\cite[Theorem~IV.1.7]{SimpsonSOSOA}.

Let $\Psi$ be a code for $f$.  Define a new code $\Phi$ by enumerating $\ball{a}{r} \vcode \Phi q$ whenever there are balls $(\ball{b_i}{s_i} : i < n)$ and rationals $\la t_i : i < n \ra$ such that $t_i > q$ and $\ball{b_i}{s_i}\vcode \Psi t_i$ for each $i < n$ and $\overline{\ball{a}{r}} \subseteq \bigcup_{i < n}\ball{b_i}{s_i}$.

Suppose that $\ball{a}{r} \vcode\Phi q$.  Then $\overline{\ball{a}{r}}$ is covered by balls $(\ball{b_i}{s_i} : i < n)$ where $f(x) > q$ for all $x \in \ball{b_i}{s_i}$ and all $i < n$.  Thus $f(x) > q$ for all $x \in \overline{\ball{a}{r}}$.  Conversely, suppose that $f(x) > q$ for all $x \in \overline{\ball{a}{r}}$.  Then any enumeration $(\ball{b_i}{s_i} : i \in \Nb)$ of every ball $\ball{b_i}{s_i}$ for which there is a $t_i \in \Qb$ with $t_i > q$ and $\ball{b_i}{s_i} \vcode \Psi t_i$ is an open cover of $\overline{\ball{a}{r}}$ and hence (essentially by \cite[Theorem~IV.1.6]{SimpsonSOSOA}) has a finite subcover.  Thus there are a sequence $(\ball{b_i}{s_i} : i < n)$ of balls and a sequence $\la t_i : i < n \ra$ of rationals witnessing that $\ball{a}{r} \vcode \Phi q$.

We have shown that, for every $\ball{a}{r} \subseteq \spc X$ and every $q \in \Qb$, $\ball{a}{r} \vcode \Phi q$ if and only if $(\forall x \in \overline{\ball{a}{r}})(f(x) > q)$.  Using this fact, it is straightforward to verify that $\Phi$ is also a code for $f$ in the sense of Definition~\ref{def-lsc}.\\

\noindent \ref{itHonestTwo} Work in $\aca$.  Let $\Phi$ be a code for $f$ as in (i).  Then, define a code $\Gamma$ by setting $\ball{a}{r} \vcode \Gamma q$ whenever it is the case that $\ball{b}{s} \vcode \Phi t$ for every $t < q$ and every $\ball{b}{s} \subseteq \spc X$ with $\ball{b}{s} \ballsub \ball{a}{r}$.

Suppose that $x \in \ball{a}{r}$ and $\ball{a}{r} \vcode \Gamma q$.  Let $\ball{b}{s}$ be such that $x \in \ball{b}{s}$ and $\ball{b}{s} \ballsub \ball{a}{r}$.  Then $\ball{b}{s} \vcode \Phi t$ for every $t < q$, which means that $f(x) > t$ for every $t < q$.  Hence $f(x) \geq q$.  Conversely, suppose that $f(x) \geq q$ for all $x \in \ball{a}{r}$.  If $\ball{b}{s} \subseteq \spc X$ and $t \in \Qb$ satisfy $t < q$ and $\ball{b}{s} \ballsub \ball{a}{r}$, then $f(x) \geq q > t$ for all $x \in \overline{\ball{b}{s}}$ because $\overline{\ball{b}{s}} \subseteq \ball{a}{r}$.  Thus $\ball{b}{s} \vcode \Phi t$.  Hence $\ball{b}{s} \vcode \Phi t$ whenever $t < q$ and $\ball{b}{s} \ballsub \ball{a}{r}$, so $\ball{a}{r} \vcode \Gamma q$.

We have shown that $\ball{a}{r} \vcode \Gamma q$ if and only if $(\forall x \in \ball{a}{r})(f(x) \geq q)$.  So $\Gamma$ is an honest code for $f$.\\

\noindent \ref{itHonestThree}  In $\aca$, we can define a code $\Phi$ by $\ball{a}{r} \vcode \Phi q$ if and only if $f(b) \geq q$ for all $b \in \ball{a}{r} \cap X$.  One readily checks that $\Phi$ is an honest code for a lower semi-continuous $g$ that is equal to $f$.\\

\noindent \ref{itHonestFour}  In $\pica$, we can directly define an honest code $\Phi$ for $f$ by setting $\ball{a}{r} \vcode \Phi q$ if and only if $(\forall x \in \ball{a}{r})(f(x) \geq q)$.
\end{proof}

\section{Continuous envelopes}

Honestly-coded lower semi-continuous functions facilitate the calculation of infima, which helps us approximate lower semi-continuous functions by continuous ones.
To be precise, lower semi-continuous functions bounded below can be written as the increasing limit of continuous functions using the following construction.

\begin{Definition}
Given a potential $\fnc$ on a complete separable metric space $\spc X$ and an $\alpha \in \Rb_{>0}$, define the {\em lower $\alpha$-envelope of $\fnc$} by
\[\fnc_\alpha(x)=\inf_{y\in \spc X}(\fnc(y)+\alpha d(x,y)).\]
\end{Definition}

Although not needed for our purposes, it is instructive to observe that if $\fnc \colon \spc X \to \overline{\Rb}_{\geq 0}$ is a potential, then $ \fnc_n$ converges pointwise to $\fnc$ as $n \to \infty$.
What we do need to prove is that continuous envelopes exist, and this can typically not be done within a weak theory.  The construction of envelopes hinges on the following more general lemma.

\begin{Lemma}[$\aca$]\label{lem-EnvelopeHelper}
Let $\spc X$ and $\spc Y$ be complete separable metric spaces, let $h \colon \spc X \times \spc Y \imp \Rb_{\geq 0}$ be uniformly continuous, and let $\fnc \colon \spc Y \imp \overline{\Rb}_{\geq 0}$ be lower semi-continuous and honestly-coded with non-empty support.  Then there is a uniformly continuous function $g \colon \spc X \imp \Rb_{\geq 0}$ such that
\[(\forall x \in \spc X)  [g(x) = \inf_{y \in \spc Y}(h(x,y) + \fnc(y))].\]
\end{Lemma}

\begin{proof}
Let $\Phi$ be a code for $h$, and let $\Psi$ be an honest code for $\fnc$.  We first show that if $x \in \spc X$, then $\inf_{y \in \spc Y}(h(x,y) + \fnc(y))$ indeed exists, which is a consequence of the following Claim.

\begin{Claim}
Let $x \in \spc X$ and $q \in \Qb$.  Then there is a $y \in \spc Y$ such that $h(x,y) + \fnc(y) < q$ if and only if there are $\ball{\langle a,b\rangle}{r}\subseteq \spc X\times \spc Y$, $(u,v) \subseteq \Rb$, and $p \in \Qb$ such that $x \in \ball{a}{r}$, $\ball{\langle a,b\rangle}{r} \fcode \Phi (u,v)$, $\neg (\ball{b}{r} \vcode \Psi p)$, and $v + p < q$.
\end{Claim}

\begin{proof}[Proof of Claim]
Be aware that $\ball{\langle a,b\rangle}{r}=\ball{a}{r}\times\ball{b}{r}$ by Definition~\ref{def-ProdSpace}.
Suppose that $y \in \spc Y$ is such that $h(x,y)  + \fnc(y) < q$.  Let $u, v, p \in \Qb$ be such that $h(x,y) \in (u,v)$, $p > f(y)$, and $v + p < q$.  Then there must be a $\ball{\langle a,b\rangle}{r} \subseteq \spc X\times\spc Y$ such that $\langle x,y\rangle \in \ball{\langle a,b\rangle}{r}$ and $\ball{\langle a,b\rangle}{r} \fcode \Phi (u,v)$.  Furthermore, since $\Psi$ is a code for $f$, it must be that $\neg(\ball{b}{r} \vcode \Psi p)$ because $y \in \ball{b}{r}$ but $\fnc(y) < p$.  Thus $\ball{\langle a,b\rangle}{r}$, $(u,v)$, and $p$ are as required.

Conversely, suppose that there are such $\ball{\langle a,b\rangle}{r}$, $(u,v)$, and $p$.  Because $\Psi$ is honest and $\neg(\ball{b}{r} \vcode \Psi p)$, there must be a $y \in \ball{b}{r}$ such that $\fnc(y) < p$.  Also, $h(x,y) \leq v$ because $\la x, y \ra \in \ball{a}{r} \times \ball{b}{r}=\ball{\langle a,b\rangle}{r}$ and $\ball{\langle a,b\rangle}{r} \fcode \Phi (u,v)$.  Hence $h(x,y) + \fnc(y) < v + p < q$.
\end{proof}

Thus, given $x \in \spc X$, let $Q$ be the set of all $q \in \Qb$ for which there are $\la a,b,r \ra \in X \times Y\times \Qb$ and $u,v,p \in \Qb$ such that $\ball{\langle a,b\rangle}{r}$, $(u,v)$, and $p$ witness that there is a $y \in \spc Y$ such that $h(x,y) + \fnc(y) < q$ as in the Claim. The set $Q$ is bounded below by $0$ and is non-empty since $\lscsupp(f) \not = \varnothing$, so $\inf Q$ exists (essentially by \cite[Theorem~III.2.2]{SimpsonSOSOA}), and, by the Claim, $\inf Q = \inf_{y \in \spc Y}(h(x,y) + \fnc(y))$.  Henceforth, for each $x \in \spc X$, let $\alpha_x$ denote $\inf_{y \in \spc Y}(h(x,y) + \fnc(y))$.

We can make the above argument uniformly for all $a \in X$ by letting $A \subseteq X \times \Qb$ be the set of all $\la a, q \ra$ for which there are $r \in \Qb$, $b \in Y$, and $u,v,p \in \Qb$ such that $\ball{\langle a,b\rangle}{r}$, $(u,v)$, and $p$ witness that there is a $y \in \spc Y$ such that $h(a,y) + \fnc(y) < q$.  Then from $A$ we can define the sequence $(\alpha_a : a \in X)$ because, given a sequence of sets of rationals all bounded from below, $\aca$ suffices to produce the corresponding sequence of infima (that is $\aca$ proves item~4 of \cite[Theorem~III.2.2]{SimpsonSOSOA} uniformly).

Define a code $\Gamma$ for a continuous partial function $g \colon \spc X \imp \Rb$ by defining $\ball{a}{r} \fcode \Gamma (u,v)$ if and only if $(\forall c \in \ball{a}{r} \cap X)[\alpha_c \in (u,v)]$.  It is easy to see that $\Gamma$ satisfies the requirements of Definition~\ref{def-continuous}.  We need to show that $g$ is indeed uniformly continuous on all of $\spc X$ and that $(\forall x \in \spc X)(g(x) = \alpha_x)$.

\begin{Claim}
Let $\varepsilon, \delta \in \Qb_{>0}$ be such that
\begin{align*}
(\forall \la x_0, y_0 \ra, \la x_1, y_1 \ra \in \spc{X} \times \spc{Y})(d_{\spc{X} \times \spc{Y}}(\la x_0, y_0 \ra, \la x_1, y_1 \ra) < \delta \imp |h(x_0,y_0) - h(x_1,y_1)| \leq \varepsilon).
\end{align*}
Let $x_0, x_1 \in \spc X$.  Then $d_{\spc X}(x_0, x_1) < \delta \imp |\alpha_{x_0} - \alpha_{x_1}| \leq \varepsilon$.
\end{Claim}

\begin{proof}[Proof of Claim]
We show that $(\forall \eta \in \Qb_{>0})(\alpha_{x_1} \leq \alpha_{x_0} + \varepsilon + \eta)$, which implies that $\alpha_{x_1} \leq \alpha_{x_0} + \varepsilon$.  By a symmetric argument, we also have that $\alpha_{x_0} \leq \alpha_{x_1} + \varepsilon$, which gives the desired $|\alpha_{x_0} - \alpha_{x_1}| \leq \varepsilon$.

Thus let $\eta \in \Qb_{>0}$.  Let $y \in \spc Y$ be such that $h(x_0, y) + \fnc(y) \leq \alpha_{x_0} + \eta$.  Then
\begin{align*}
\alpha_{x_1} \leq h(x_1, y) + \fnc(y) = (h(x_1, y) - h(x_0, y)) + (h(x_0, y) + \fnc(y)) \leq \varepsilon + \alpha_{x_0} + \eta,
\end{align*}
where the last inequality is by the choice of $y$ and the fact that $d_{\spc{X} \times \spc{Y}}(\la x_0, y \ra, \la x_1, y \ra) = d_{\spc X}(x_0, x_1) < \delta$.
\end{proof}

Given $x \in \spc X$ and $\varepsilon \in \Qb_{>0}$, let $\delta \in \Qb_{>0}$ be such that
\begin{align*}
(\forall \la x_0, y_0 \ra, \la x_1, y_1 \ra \in \spc{X} \times \spc{Y})(d_{\spc{X} \times \spc{Y}}(\la x_0, y_0 \ra, \la x_1, y_1 \ra) < \delta \imp |h(x_0,y_0) - h(x_1,y_1)| \leq \nicefrac \varepsilon 4),
\end{align*}
and let $a \in X$ be such that $x \in \ball{a}{\delta}$.  Let $(u,v) \subseteq \Rb$ be such that $[\alpha_a - \nicefrac \varepsilon 4, \alpha_a + \nicefrac \varepsilon 4] \subseteq (u,v)$ and $v-u < \varepsilon$.  The Claim tells us that if $c \in \ball{a}{\delta} \cap X$, then $\alpha_c \in [\alpha_a - \nicefrac \varepsilon 4, \alpha_a + \nicefrac \varepsilon 4] \subseteq (u,v)$, which means that $\ball{a}{\delta} \fcode \Gamma (u,v)$.  Likewise, the Claim implies that $\alpha_x \in (u,v)$.  Thus for every $x \in \spc X$ and every $\varepsilon \in \Qb_{>0}$, there are a ball $\ball{a}{\delta} \subseteq \spc X$ and a $(u,v) \subseteq \Rb$ such that $x \in \ball{a}{\delta}$, $\alpha_x \in (u,v)$, $v-u < \varepsilon$, and $\ball{a}{\delta} \fcode \Gamma (u,v)$.  Thus $g(x)$ is defined and equals $\alpha_x$ for all $x \in \spc X$.  Furthermore, the Claim clearly implies that $g$ is uniformly continuous.
\end{proof}

\begin{Lemma}[$\aca$]\label{lem-EnvelopeExists}
Let $\spc X$ be a complete separable metric space, let $\fnc \colon \spc X \imp \overline{\Rb}_{\geq 0}$ be lower semi-continuous and honestly coded with non-empty support, and let $\alpha \in \Rb_{\geq 0}$.  Then there is a uniformly continuous function $\fnc_\alpha \colon \spc X \imp \Rb_{\geq 0}$ such that, for all $x\in \spc X$,
\begin{align*}
\fnc_\alpha(x) = \inf_{y \in \spc X}(\fnc(y) + \alpha d(x,y)).
\end{align*}
\end{Lemma}

\begin{proof}
This follows from Lemma~\ref{lem-EnvelopeHelper} because the function $h \colon \spc X \times \spc X \imp \Rb_{\geq 0}$ given by $h(x,y) = \alpha d_{\spc X}(x,y)$ is uniformly continuous and bounded below by $0$.
\end{proof}

\section{Critical points of continuous functions}\label{SecFormalEVP}

In this section we formalize Ekeland's variational principle and show that it has natural restrictions provable in weak systems.  Recall that a \emph{potential} on a complete separable metric space $\spc X$ is a lower semi-continuous function $\fnc \colon \spc X \imp \overline{\Rb}_{\geq 0}$. Recall also that \emph{$\varepsilon$-critical points} were defined in Definition \ref{DefCritical} and that this definition may readily be formalized in $\rca$ using coded lower semi-continuous functions.
With this, we are ready to define our formalized version (or, version{\em s}) of Ekeland's variational principle.

\begin{Definition}\label{DefFormalCritical}
Given definable classes $\mathfrak X$ of coded complete separable metric spaces and $\mathfrak F$ of coded potentials, the {\em (formalized) free variational principle ($\fvp$) for $\spc X\in\mathfrak X$ and $\fnc\in\mathfrak F$} is the statement that, if $\spc X\in\mathfrak X$ and $\fnc \colon \spc X \to \overline{\Rb}_{\geq 0}$ in $\mathfrak F$ is such that $ \lscsupp (f) \not = \varnothing$, then for every $\varepsilon > 0$ there is an $x_\ast \in \lscsupp (f)$ such that
\[ \forall x\in \lscsupp (f) \Big ( \big ( \varepsilon d(x_\ast,x) \leq  \fnc (x_\ast) - \fnc(x) \big ) \rightarrow x = x_\ast \Big) .\]
The {\em localized variational principle} ($\lvp$) is defined similarly, except that $x_0 \in \lscsupp (f)$ is given and $x_\ast$ is chosen so that $ \varepsilon d(x_0, x_\ast)  \leq f(x_0) - f(x_\ast)$.

When not mentioned, we assume that $\mathfrak X$ is the class of all coded complete separable metric spaces and $\mathfrak F$ is the class of all coded potentials.  Given $c>0$, to indicate that $f(x) \leq c$ for all $x \in \lscdom (f)$, we will say that $f$ is {\em $c$-bounded,} and if we want to fix the value of $\varepsilon$ to $c$, we call the statement the {\em free/localized variational principle for $\varepsilon = c$}.  The free/localized variational principles for $\varepsilon = c$ and $f$ $c$-bounded will be denoted the $c$-$\fvp$ and the $c$-$\lvp$, respectively.
\end{Definition}

Critical points are sometimes also called {\em pseudo-minima,} as they may serve as approximate minima of functions that do not actually attain their minimum value.

\begin{Lemma}[$\rca$]\label{lemmCritApproxMin}
If $\spc X$ is any complete separable metric space and $f\colon \spc X \to \overline{\Rb}_{\geq 0}$ is lower semi-continuous with non-empty support, then for any $x_\ast \in \spc X$, $f$ attains its minimum at $x_\ast $ if and only if $x_\ast$ is an $\varepsilon$-critical point of $f$ for all $\varepsilon > 0$.
\end{Lemma}

\begin{proof}
It is straightforward to see that if $x \not = x_\ast$ and $ \varepsilon d(x,x_\ast) \leq f(x_\ast) - f(x) $ then $f(x) < f(x_\ast)$, hence if $x_\ast$ is not $\varepsilon$-critical it cannot be a minimum. Conversely, if $f$ does not attain its minimum at $x_\ast$, then we may choose $x \in \spc X$ with $f(x) < f(x_\ast)$. By choosing $\varepsilon$ small enough, we obtain $ \varepsilon d(x,x_\ast)  \leq f(x_\ast) - f(x)$.
\end{proof}

The class $\mathfrak F$ will contain either continuous or lower semi-continuous functions.
Important special cases for us are the $1$-$\fvp$ and the $1$-$\lvp$, which on occasion are equivalent to the $\fvp$ and $\lvp$, respectively.

\begin{Lemma}[$\rca$]\label{lemEpsRest}
Let $\mathfrak X$ be a class of complete separable metric spaces and $\mathfrak F$ be either the class of continuous or of lower semi-continuous potentials.
\begin{enumerate}

\item \label{itEpsUnbounded} The $\fvp$ and the $\lvp$ with $\spc X \in \mathfrak X$ and  $\fnc \in \mathfrak F $ are equivalent, respectively, to the $\fvp$ and the $\lvp$ with $\varepsilon = 1$, $\spc X \in \mathfrak X$, and  $\fnc \in \mathfrak F $.

\item \label{itEpsBigClass} If $\mathfrak X$ is either the class of all complete separable metric spaces or of all compact metric spaces then the $\fvp$ and the $\lvp$ with $\spc X \in \mathfrak X$ and  $\fnc \in \mathfrak F $ bounded are equivalent, respectively, to the $1$-$\fvp$ and the $1$-$\lvp$ with $\spc X \in \mathfrak X$ and $\fnc \in \mathfrak F $.

\item \label{itEpsSingletons} If $\spc X$ is
\begin{enumerate*}

\item\label{itBanachSpace}  a Banach space,

\item\label{itBaireSpace} the Baire space,

\item\label{itCantorSpace} the Cantor space, or

\item \label{itBallRn} a closed ball in $\Rb^n$ (with any norm),

\end{enumerate*}
then the $\fvp$ and the $\lvp$ for $\spc X$ and bounded, lower semi-continuous potentials are equivalent, respectively, to the $1$-$\fvp$ and the $1$-$\lvp$ for $\spc X$ and lower semi-continuous potentials.
\end{enumerate}
\end{Lemma}

\begin{proof}
\eqref{itEpsUnbounded} Replace $f$ by $f' = \nicefrac f \varepsilon$ and note that any $1$-critical point for $f'$ is $\varepsilon$-critical for $f$.
\medskip

\noindent \eqref{itEpsBigClass} First replace $f$ by $f' = \nicefrac fb$, so that $f'$ is bounded by $1$ and any $\nicefrac \varepsilon b$-critical point for $f'$ is an $\varepsilon$-critical point for $f$.
As we can no longer freely scale $f'$, we instead scale the metric and consider $d' = \nicefrac {\varepsilon d} b $.
Then the reader can verify that any $1$-critical point for $f'$ with respect to $d'$ is an $\varepsilon$-critical point for $f$ with respect to $d$.\medskip

\noindent \eqref{itEpsSingletons}
The general idea is that in all of these cases, the proof of \eqref{itEpsBigClass} can be simulated (in a possibly {\em ad-hoc} way) without modifying the space $\spc X$.
For the sake of illustration we sketch the proof in the case where $\spc X$ is a closed ball in $\Rb^n$---we give an informal argument and let the reader verify that it can be carried out in $\rca$.
The $\fvp$ and $\lvp$ cases are similar, so we focus on the $\fvp$.

Assume that the $1$-$\fvp$ holds for $\spc X$.
Without loss of generality we may assume that $\spc X = \overline{\ball 0\rho}$ for some $\rho>0$.
Let $f$ be a bounded lower semi-continuous function on $\spc X$ with non-empty domain, and let $\varepsilon>0$. Without loss of generality we may assume that $f$ is bounded by $\nicefrac 12$ (adjusting $\varepsilon$ if needed).

Let $A\subseteq {\ball 0 \rho}\cap \Qb^n$ be finite and such that every point in $\ball 0 \rho$ is within distance $\nicefrac 14$ of some point in $A$.  Let $\delta > 0$ be small enough so that $  \delta   < \min \{ \varepsilon\rho,\nicefrac 14\} $ and $d(a,a')> 2\delta$ whenever $a,a'\in A$ are distinct. Define $g\colon \spc X \to \overline{\Rb}_{\geq 0}$ by 
\[
g(x)=
\begin{cases}
f(\frac \rho \delta (x-a)) & \text{if $x\in \overline{\ball a\delta}$ for some $a \in A$}\\
1&\text{otherwise}.
\end{cases}
\]
By the $1$-$\fvp$, let $x_\ast$ be a $1$-critical point of $g$.
We claim that $x_\ast\in \overline{\ball a\delta}$ for some $a \in A$.
If not, then $g(x_\ast) = 1$ by the definition of $g$.
Let $a\in A$ minimize $d(x_\ast,a)$, and note that $d(x_\ast,a) <\nicefrac 14$ by our choice of $A$.
Let $y_0\in {\rm supp} ( f )$.
Then $x_0:=\frac \delta \rho y_0 + a \in \overline{\ball a\delta}$ and $g(x_0) = f(y_0) \leq \nicefrac 12$, so that
\[g(x_\ast) - g(x_0 ) \geq 1-\nicefrac 12 =\nicefrac 14 + \nicefrac 14 \geq d(x_0,a) + d(a,x_\ast) \geq d(x_0,x_\ast). \]
This contradicts that $x_\ast$ is a $1$-critical point of $g$.
Finally, we claim that $y_\ast := \frac \rho \delta (x_\ast-a)$ is an $\varepsilon$-critical point of $f$.
Let $y \in {\rm supp} ( f ) $ with $y \neq x_\ast$  be arbitrary, and let $x:=\frac \delta \rho y + a \in \ball a\delta $. Then since $x_\ast$ is a $1$-critical point of $g$,
\[g( x _\ast) - g(x) < d( x_\ast,x ) ,\]
which by the definition of $g$ becomes
\[f(y_\ast) - f(y) < d\left( \frac \delta \rho y_\ast + a, \frac \delta \rho y  + a \right) = \frac \delta \rho d(y_\ast,y ) < \varepsilon d(y_\ast,y ), \]
as needed.

The proof that the $1$-$\lvp$ implies the bounded $\lvp$ is similar, except that it is more convenient to replace $A$ by the singleton $\{0\}$.
The proofs of the other items are analogous, with the main difference being the choice of $g$.
Let us assume that $f$ is bounded by $1$.
For \eqref{itBanachSpace} we set $g(x) = f\left(\frac{1}{\varepsilon}x\right)$.
For items \eqref{itBaireSpace} and \eqref{itCantorSpace} we use the observation that $d(\sigma^\smf x, \sigma^\smf y) = 2^{-|\sigma|}d(x, y)$ for every $\sigma$, $x$, and $y$.
Let $n$ be large enough so that $2^{-n} < \varepsilon$.  For the $\fvp$ define
$g(\sigma^\smf y) =  f(y)$, where $|\sigma| = n$. That is, given $x$, $g(x)$ chops off $x \rest n$ and applies $f$ to the remaining sequence $y$.
For the $\lvp$, we define $g((0^n)^\smf y) =  f(y)$ and $g(x) = 1$ if $x$ does not begin with $n$ zeroes.
We leave to the reader to check that in each of this cases, any $1$-critical point for $g$ induces an $\varepsilon$-critical point for $f$, satisfying the localization condition when appropriate.
\end{proof}

With Lemma~\ref{lemEpsRest} in mind, we often take $\varepsilon = 1$ or, in the bounded case, $f$ bounded by $1$, and we write {\em critical point} instead of {\em $1$-critical point}.
Note however that if we fix $\spc X$ beforehand, the $\fvp$ or $\lvp$ for $\spc X$ with bounded $f$ are not necessarily equivalent to the $1$-$\fvp$ or the $1$-$\lvp$ for $\spc X$, respectively, since we may not be able to perform the required scaling transformations without modifying the space.
Claim \eqref{itEpsSingletons} gives a few examples where such transformations may be performed, and while the list is by no means meant to be exhaustive, we conjecture that a space $\spc X$ can be found so that the $1$-$\fvp$ and the $\fvp$ for $\spc X$ are not provably equivalent in $\rca$.

Note also that in the cases \eqref{itBanachSpace}-\eqref{itCantorSpace}, if the function $f$ is continuous, then so is $g$. Thus in these cases, the claim also holds for the respective variational principles restricted to continuous functions.
Our construction for \eqref{itBallRn} does not preserve continuity, but it should also be possible to define some $\tilde g$ which is continuous if $f$ was (similar to the $2$-envelope of $g$). Note however that some care would be required to construct $\tilde g$ in $\rca$.
In any case, the current formulation of the lemma suffices for our purposes.
 
Let us now begin our analysis by showing that $\wkl$ suffices to construct critical points of continuous functions on compact spaces.

\begin{Proposition}[$\wkl$]\label{propWKLtoFVP}
The $\fvp$ holds for compact $\spc X$ and continuous $f$; in fact, in this setting $f$ attains its minimum.
\end{Proposition}

\begin{proof}
$\wkl$ proves that every continuous function from a compact complete separable metric space to $\Rb$ attains a minimum value (see~\cite[Theorem~IV.2.2]{SimpsonSOSOA} for the version of this fact with `maximum' in place of `minimum').  Thus let $x_* \in \spc X$ be a point at which $\fnc$ attains its minimum value. Then $x_*$ is a critical point of $\fnc$ by Lemma \ref{lemmCritApproxMin}.
\end{proof}

When working over $\aca$, we may drop the assumption that $\spc X$ is compact.  Our proof is based on those in~\cites{Brezis1976,Du16}.

\begin{Theorem}[$\aca$]\label{thm-ContinuousCritPt}
The $\fvp$ holds for arbitrary $\spc X$ and continuous $f$.
\end{Theorem}

\begin{proof}
Let $\spc X = \widehat{X}$ be a complete separable metric space with metric $d$, and in view of Lemma \ref{lemEpsRest}.\ref{itEpsUnbounded}, assume that $\varepsilon = 1$. First we use $\aca$ to collect some information in order to implement the proof.  Define $S \subseteq X \times X \times \Qb_{>0}$ by
\begin{align*}
S = \{(a, b, q)  \in X \times X \times \Qb_{>0} : d(a,b) < \fnc(a) - \fnc(b) + q\},
\end{align*}
and for $a \in X$ and $q \in \Qb_{>0}$, write $S(a,q)$ for $\{b \in X : (a, b, q) \in S\}$.  Define a sequence of reals $(u_{a,q} : a \in X \andd q \in \Qb_{>0})$ by $u_{a,q} = \inf_{b \in S(a,q)}\fnc(b)$ for each $a \in X$ and $q \in \Qb_{>0}$, which always exists because $\fnc$ is bounded from below and $S(a,q)$ is non-empty (as it contains $a$).  Next, define a sequence of reals $(r_a : a \in X)$ by $r_a = \sup_{q \in \Qb_{>0}}u_{a,q}$ for each $a \in X$, which always exists because $(\forall q \in \Qb_{>0})(u_{a,q} \leq \fnc(a))$.  That the sequences $(u_{a,q} : a \in X \andd q \in \Qb_{>0})$ and $(r_a : a \in X)$ can be defined in $\aca$ follows from the appropriately uniformized version of the (1)$\Imp$(4) direction of~\cite[Theorem~III.2.2]{SimpsonSOSOA} (and the analogous implication with `greatest lower bound' in place of `least upper bound' in (4)).  Observe that $(\forall a \in X)(\forall q_0, q_1 \in \Qb_{>0})(q_0 \leq q_1 \imp u_{a, q_1} \leq u_{a, q_0})$, which means that $(\forall a \in X)(r_a = \lim_{q \imp 0^+}u_{a,q})$.  Finally, define a function $R \colon X \times \Qb_{>0} \imp \Qb$ describing the rates of convergence of the sequences $(u_{a,q} : q \in \Qb_{>0})$ for each $a \in X$ by letting $R(a,p) = 2^{-n}$, where $n$ is least such that $(\forall q \in \Qb_{>0})(q \leq 2^{-n} \imp |r_a - u_{a,q}| < p)$.

Now we define a sequence $(a_n : n \in \Nb)$ of points in $X$ converging to a critical point of $\fnc$ and a helper sequence $(q_n : n \in \Nb)$ of points in $\Qb_{>0}$.  Let $a_0$ be the first point (that is, the point with the least code) in $X$, and let $q_0 = 2^{-2}R(a_0, 1)$.  Given $(a_i : i \leq n)$ and $(q_i : i \leq n)$, let $a_{n+1}$ be the first point in $S(a_n, q_n)$ such that $\fnc(a_{n+1}) < u_{a_n, q_n} + 2^{-(n+1)}$, and let $q_{n+1} = 2^{-(n+3)}\prod_{i \leq n+1}R(a_i, 2^{-i})$.  The important property of $(q_n : n \in \Nb)$ is that, for every $n \in \Nb$,
\begin{align*}
\sum_{k \geq n} q_k < \min\{2^{-n}, R(a_n, 2^{-n})\},
\end{align*}
which holds because $(\forall n)(\forall k \geq n)[q_k \leq 2^{-(k+2)}R(a_n, 2^{-n})]$.

As the sequence $(\fnc(a_n) : n \in \Nb)$ is bounded from below, let $\ell = \liminf_{n \in \Nb}\fnc(a_n)$ (which exists by~\cite[Lemma~III.2.1]{SimpsonSOSOA} with `$\liminf$' in place of `$\limsup$').  Observe that $\liminf_{n \in \Nb} u_{a_n, q_n} = \ell$ as well because $\forall n(u_{a_n, q_n} \leq \fnc(a_{n+1}) < u_{a_n, q_n} + 2^{-(n+1)})$.  Let $i_0 < i_1 < i_2 < \dots$ be an increasing sequence of indices such that $\lim_{n \in \Nb}\fnc(a_{i_n}) = \ell$.  We show that the corresponding sequence $(a_{i_n} : n \in \Nb)$ is Cauchy.  Given $M \in \Nb$, let $N > M$ be large enough so that $\forall n,m (N < n < m \imp |\fnc(a_{i_n}) - \fnc(a_{i_m})| \leq 2^{-(M+1)})$.  Then, for every $n,m \in \Nb$ with $N < n < m$, we have that
\begin{align}
\nonumber d(a_{i_n}, a_{i_m})
&\leq \sum_{k = i_n}^{i_m - 1}d(a_k, a_{k+1})\\
\nonumber &\leq \sum_{k = i_n}^{i_m - 1} [\fnc(a_k) - \fnc(a_{k+1}) + q_k]\\
\label{ineqAIM} &\leq \fnc(a_{i_n}) - \fnc(a_{i_m}) + \sum_{k \geq i_n}q_k\\
\nonumber &\leq 2^{-(M+1)} + 2^{-(M+1)} \leq 2^{-M}.
\end{align}
In the above expression, the first inequality is the triangle inequality, the second inequality is because $a_{k+1}$ is chosen from $S(a_k, q_k)$ for every $k \in \Nb$, the third inequality is obtained by canceling the appropriate terms in the telescoping sum, and the fourth inequality is by the choice of $N$ and by recalling that $\sum_{k \geq i_n}q_k < 2^{-i_n}$ (and that $i_n \geq n > N > M$).  Thus $(a_{i_n} : n \in \Nb)$ is Cauchy.  By the (1)$\Imp$(3) direction of~\cite[Theorem~III.2.2]{SimpsonSOSOA}, $(a_{i_n} : n \in \Nb)$ converges to some $x_* \in \spc X$.

We show that $x_*$ is a critical point of $\fnc$.  By the continuity of $\fnc$, $\fnc(x_*) = \lim_{n \in \Nb}\fnc(a_{i_n}) = \ell$.  We show that if $y \in \spc X$ is such that $d(x_* ,y) \leq \fnc(x_*) - \fnc(y)$, then $\fnc(y) = \ell$ as well.  Thus consider a $y \in \spc X$ such that $d(x_* ,y) \leq \fnc(x_*) - \fnc(y)$.  Clearly $\fnc(y) \leq \ell$ because otherwise we would have the contradiction $d(x_*,y) < 0$.  So we show that $\fnc(y) \geq \ell$.  To do this, fix $n \in \Nb$.  We see that $d(a_{i_n}, x_*) \leq \fnc(a_{i_n}) - \fnc(x_*) + \sum_{k \geq i_n}q_k$ by considering the inequality \eqref{ineqAIM} as $m \imp \infty$.  Therefore,
\begin{align*}
d(a_{i_n}, y) &\leq d(a_{i_n}, x_*) + d(x_*, y)\\
&\leq \fnc(a_{i_n}) - \fnc(x_*) + \sum_{k \geq i_n}q_k + \fnc(x_*) - \fnc(y)\\
&= \fnc(a_{i_n}) - \fnc(y) + \sum_{k \geq i_n}q_k.
\end{align*}
Thus $d(a_{i_n}, y) < \fnc(a_{i_n}) - \fnc(y) + R(a_{i_n}, 2^{-i_n})$ because $\sum_{k \geq i_n}q_k < R(a_{i_n}, 2^{-i_n})$.  It follows that $\fnc(y) \geq u_{a_{i_n}, R(a_{i_n}, 2^{-i_n})}$ because $y$ is in the open set
\[\MC U = \{x \in \spc X : d(a_{i_n}, x) < \fnc(a_{i_n}) - \fnc(x) + R(a_{i_n}, 2^{-i_n})\}\]
and $S(a_{i_n}, R(a_{i_n}, 2^{-i_n}))$ is dense in $\MC U$.  Observe also that $q_{i_n} \leq R(a_{i_n}, 2^{-i_n})$, which implies that $u_{a_{i_n}, q_{i_n}} - u_{a_{i_n}, R(a_{i_n}, 2^{-i_n})} < 2^{-i_n}$.  Therefore $\fnc(y) > u_{a_{i_n}, q_{i_n}} - 2^{-i_n}$.  So we have shown that $\forall n (\fnc(y) > u_{a_{i_n}, q_{i_n}} - 2^{-i_n})$.  As $(u_{a_{i_n}, q_{i_n}} : n \in \Nb)$ is a subsequence of $(u_{a_n, q_n} : n \in \Nb)$, it follows that $\liminf_{n \in \Nb} u_{a_{i_n}, q_{i_n}} \geq \liminf_{n \in \Nb} u_{a_n, q_n} = \ell$.  Thus we conclude that $\fnc(y) \geq \ell$, as desired.  This completes the proof that if $d(x_*, y) \leq \fnc(x_*) - \fnc(y)$, then $\fnc(y) = \fnc(x_*) = \ell$.  So if $y \in \spc X$ satisfies $d(x_*, y) \leq \fnc(x_*) - \fnc(y)$, then $y = x_*$.  Hence $x_*$ is a critical point of $\fnc$.
\end{proof}

\section{Critical points of arbitrary potentials}\label{secCPAP}

Now let us consider the case where $f$ is a possibly discontinuous potential.
As it turns out, this general case can be reduced to the continuous case by using envelopes.
To be precise, in order to find an $\alpha$-critical point of a lower semi-continuous potential $\fnc$, it suffices to find an $\alpha$-critical point of the $\beta$-envelope $\fnc_\beta$ for any $\beta > \alpha$.

\begin{Lemma}[$\rca$]\label{lem-EnvelopeCritPt}
Let $\spc X$ be a complete separable metric space, $ 0 < \alpha <\beta$, $\fnc \colon \spc X \imp \overline{\Rb}_{\geq 0}$ be a potential with $\lscsupp (f) \not = \varnothing$, and suppose that there is a continuous function $\fnc_\beta \colon \spc X \imp \Rb_{\geq 0}$ such that, for every $x \in \spc X$,
\begin{align*}
\fnc_\beta(x) = \inf_{y \in \spc X}(\fnc(y) + \beta d(x,y)).
\end{align*}
Then, for any $x_\ast \in \spc X$,

\begin{enumerate}

\item if $x_* $ is an $\alpha$-critical point of $\fnc_\beta$, then $\fnc(x_*) = \fnc_\beta(x_*)$ and $x_*$ is an $\alpha$-critical point of $\fnc$.

\item If $\fnc _\beta$ attains its minimum at $x_* $, then $\fnc$ attains its minimum at $x_*$.

\end{enumerate}

\end{Lemma}

\begin{proof}
In view of Lemma \ref{lemmCritApproxMin}, the second item is a consequence of the first: indeed, if $\fnc _\beta$ attains its minimum at $x_* \in \spc X$, then for every $\alpha < \beta$ we have that $x_\ast$ is $\alpha$-critical for $f_\beta$, so that by the first item it is $\alpha$-critical for $f$. Since $\alpha$ is arbitrary, $f$ attains its minimum at $x_\ast$.

Thus we focus on the first item.
Suppose that $x_*$ is an $\alpha$-critical point of $\fnc_\beta$; we begin by showing that $f(x_*) = f_\beta(x_*)$.
Clearly $\fnc_\beta(x_*) \leq \fnc(x_*)$, so suppose for a contradiction that $\fnc_\beta(x_*) < \fnc(x_*)$.
Let $\varepsilon = 1$ if $f(x_*) = \infty$ and $\varepsilon = \frac{\fnc(x_*) - \fnc_\beta(x_*)}{2} > 0$ otherwise, and let $\delta \in \Qb_{>0}$ be such that $\fnc(y) > \fnc_\beta (x_*) + \varepsilon$ whenever $d(x_*, y) < \delta $.  Such a $\delta$ can be obtained by letting $\Phi$ be the code for $\fnc$; letting $\ball{a}{r} \subseteq \spc X$ and $q \in \Qb_{>0}$ be such that $\ball{a}{r} \vcode \Phi q$, $x_* \in \ball{a}{r}$, and $q > \fnc_\beta (x_*) + \varepsilon$; and then taking $\delta <  r - d(a, x_*)$.
  Now let $y \in \spc X$ be such that
\[\fnc_\beta(x_*) + \min\{(\beta - \alpha ) \delta, \varepsilon\} > \fnc(y) + \beta d(x_*,y) .\]
Note that $y \neq x_*$, for otherwise $\fnc_\beta(x_*) > \fnc(x_*) - \min\{ (\beta - \alpha ) \delta, \varepsilon\} \geq \fnc_\beta (x_*) + \varepsilon$, a contradiction.  Moreover, we cannot have that $d(x_*, y) < \delta$, for otherwise
\begin{align*}
\fnc_\beta(x_*) &> \fnc(y) + \beta d(x_*, y) - \min\{ (\beta - \alpha ) \delta, \varepsilon\}\\
&> (\fnc_\beta(x_*) + \varepsilon) + \beta d(x_*, y) - \min\{ (\beta - \alpha ) \delta, \varepsilon\}\\
&\geq \fnc_\beta(x_*) + \beta d(x_*, y),
\end{align*}
where the first inequality is by the choice of $y$ and the second inequality is by the choice of $\delta$ and the assumption $d(x_*, y) < \delta$.  Thus it must be that $d(x_*, y)\geq \delta$.  Therefore,
\begin{align*}
\alpha d(x_*, y)& < \fnc_\beta(x_*) - \fnc(y) - (\beta - \alpha) d(x_*, y) +\min\{(\beta - \alpha ) \delta, \varepsilon\} \\
& \leq \fnc_\beta(x_*) - \fnc_\beta(y) - (\beta - \alpha) \delta + (\beta - \alpha) \delta = \fnc_\beta(x_*) - \fnc_\beta(y),
\end{align*}
where the first inequality is by the choice of $y$ and the second inequality is because $\fnc_\beta(y) \leq \fnc(y)$, $\delta \leq d(x_*, y)$, and $\min\{ (\beta - \alpha) \delta, \varepsilon\} \leq (\beta - \alpha) \delta$.
This means that $x_*$ is not an $\alpha$-critical point of $\fnc_\beta$, which is a contradiction.

We have established that $\fnc(x_*) = \fnc_\beta(x_*)$.  We now use this to show that $x_*$ is an $\alpha$-critical point of $\fnc$.  Assume for a contradiction that this is not the case.  Then there is a $y \in \spc X$ such that $\alpha d(x_*, y) \leq \fnc(x_*) - \fnc(y)$ but $y \neq x_*$.  Then,
\begin{align*}
\alpha d(x_*, y) \leq \fnc(x_*) - \fnc(y) \leq \fnc_\beta(x_*) - \fnc_\beta(y)
\end{align*}
because $\fnc_\beta(x_*) = \fnc(x_*)$ and $\fnc_\beta(y) \leq \fnc(y)$.  This contradicts that $x_*$ is an $\alpha$-critical point of $\fnc_\beta$.  Thus $x_*$ is indeed an $\alpha$-critical point of $\fnc$.
\end{proof}

\begin{Theorem}\label{thm-CritPtsExist}{\ }
\begin{enumerate}[label=(\roman*)]

\item\label{CritPtsExistOne} \textup{(}$\aca$\textup{)} The $\fvp$ holds for arbitrary $\spc X$ and any honestly-coded potential $f$.

\item\label{CritPtsExistTwo} \textup{(}$\aca$\textup{)} The $\fvp$ holds for compact $\spc X$ and any potential $f$; in fact, such an $f$ attains its minimum.

\item\label{CritPtsExistThree} \textup{(}$\pica$\textup{)} The $\fvp$ holds for arbitrary $\spc X$ and any potential $f$.

\end{enumerate}
\end{Theorem}

\begin{proof}
In view of Lemma \ref{lemmCritApproxMin}, we may assume that $\varepsilon = 1$; note that an honest code remains honest after scaling $f$.
\medskip

\noindent \ref{CritPtsExistOne} As $\fnc$ is assumed to be honestly-coded, the envelope $\fnc_2$ exists and is continuous by Lemma~\ref{lem-EnvelopeExists}.  The function $\fnc_2$ has a critical point by Theorem~\ref{thm-ContinuousCritPt}, and this critical point is also a critical point of $\fnc$ by Lemma~\ref{lem-EnvelopeCritPt}.
\medskip

\noindent \ref{CritPtsExistTwo} As $\spc X$ is compact, $\fnc$ can be honestly coded by Lemma~\ref{lem-HonestCodes} item~(ii).  Thus $\fnc$ has a critical point by item~\ref{CritPtsExistOne} of this theorem. In fact, $\fnc_2$ is defined and attains its minimum by Proposition \ref{propWKLtoFVP}, so that $\fnc$ also attains its minimum.
\medskip

\noindent \ref{CritPtsExistThree} Working in $\pica$, we can assume that $\fnc$ is honestly-coded by Lemma~\ref{lem-HonestCodes} item~\ref{itHonestFour}.  Thus $\fnc$ has a critical point by item~\ref{CritPtsExistOne}.
\end{proof}

Below we will show that the points in Theorem \ref{thm-CritPtsExist} are optimal.

\section{Reversals of the variational principle}\label{secReversals}

In this section we show that many of the results of \S\ref{SecFormalEVP} and \S\ref{secCPAP} reverse. Let us begin with the weakest version of the $\fvp$ we have considered. Recall from Proposition~\ref{propWKLtoFVP} that $\wkl$ suffices to prove that every continuous potential over a compact space has a critical point.  We now show that $\wkl$ is also necessary to prove this version of the $\fvp$.

\begin{Proposition}\label{propFVPtoWKL}
The $\fvp$ for continuous, bounded $f$ on the Cantor space or on $[0,1]$ implies $\wkl$ over $\rca$.
\end{Proposition}

\begin{proof} 
We work in $\rca$ and prove the contrapositive. If $\wkl$ fails, then there is an infinite binary tree $T\subseteq 2^{<\Nb}$ that has no infinite path.
Let $T^{\circ}$ be the set of leaves of $T$: $T^{\circ}=\{\sigma\in T: \sigma^\smf 0, \sigma^\smf 1\notin T\}$.
Since $T$ is infinite and has no infinite path, $T^{\circ}$ is also infinite.
For each $\sigma\in T^{\circ}$, define
\begin{align*}
A_{\sigma}=\{i<|\sigma|-1: \neg(\exists \tau \in T)(|\tau|=|\sigma|+1 \text{ and } \tau\sqsupseteq (\sigma\rest i)^\smf(1-\sigma(i+1))\}.
\end{align*}
For each $\sigma\in T$, define $\tilde \sigma\in 2^{<\Nb}$ with $|\tilde\sigma|=2|\sigma|$ as $\tilde\sigma(2i)=0$ and $\tilde\sigma(2i+1)=\sigma(i)$ if $i<|\sigma|$.
Put $\tilde T=\{\tilde \sigma:\sigma\in T\}$, $\tilde{T}^{\circ}=\{\tilde\sigma: \sigma\in T^{\circ}\}$, and put
\begin{align*}
S=\{\tau\in 2^{<\Nb}: (\forall \sigma \in T)(\tau\not\sqsubseteq \tilde\sigma) \text{ and } (\exists \sigma \in T)(\tau\rest(|\tau|-1)\sqsubset \tilde\sigma)\}.
\end{align*}
Here, $S$ is the set of all binary strings $\tau$ which move away from $\tilde T$ before reaching a member of $\tilde T^{\circ}$.  $S$ can be defined in $\rca$ because each $\tau$ need only be checked against strings $\tilde\sigma$ of length at most $|\tau|$.  The elements of $\tilde{T}^{\circ}\cup S$ are pairwise incomparable, and for every $x\in 2^{\Nb}$ there is a $\sigma\in \tilde{T}^{\circ}\cup S$ such that $x\sqsupseteq \sigma$.  Indeed, each $x$ must either reach a leaf of $\tilde T$ or move away from $\tilde T$ before that because $\tilde T$ has no infinite path.

Now, define a continuous function $f:2^{\Nb}\to [0,3]$ as follows:
\begin{align*}
 f(x)=
\begin{cases}
 2-\sum_{i\in A_{\sigma}}2^{-2i} & \mbox{if $x\sqsupseteq \tilde \sigma$ for some $\sigma\in T^{\circ}$},\\
 3 & \mbox{if $x\sqsupseteq \tau$ for some $\tau\in S$}.
\end{cases}
\end{align*}
One may easily obtain a code for $f$ by Lemma~\ref{lem-code-for-conti1} and Lemma~\ref{lem-code-for-conti2}.
We show that $f$ has no critical point, which gives the desired contradiction.
If $x\sqsupseteq \tau$ for some $\tau\in S$, then take any $\sigma\in T^{\circ}$ and $y\sqsupseteq\tilde\sigma$, and observe that
$f(x)-f(y)\ge 3-2\ge d(x,y)$.  Thus $x$ is not a critical point.
Assume instead that $x\sqsupseteq \tilde \sigma$ for some $\sigma\in T^{\circ}$.
Let $i_{0}$ be the greatest $i<|\sigma|-1$ such that $i\notin A_{\sigma}$, which exists because $T$ is infinite.  Then there is a $\sigma'\in T^{\circ}$ such that $\sigma'\sqsupseteq (\sigma\rest i_{0})^\smf(1-\sigma(i_{0}+1))$ and $|\sigma'|>|\sigma|$.
By the maximality of $i_{0}$, any $\tau\in T$ which extends $\sigma\rest (i_{0}+1)$ is shorter than $\sigma'$.
Thus, we have that $i_{0}\in A_{\sigma'}$ and that $j\in A_{\sigma}$ implies that $j\in A_{\sigma'}$ for every $j<i_{0}$ since $\sigma\rest i_{0}=\sigma'\rest i_{0}$.  Take $y\sqsupseteq \tilde\sigma'$.
Then $d(x,y)\le 2^{-2i_{0}-1}$, and therefore
\[f(x)-f(y)=-\sum_{i\in A_{\sigma}}2^{-2i}+\sum_{i\in A_{\sigma'}}2^{-2i}\ge 2^{-2i_{0}}-\sum_{i\in A_{\sigma},i>i_{0}}2^{-2i}\ge 2^{-2i_{0}-1}\ge d(x,y).\]
Thus $x$ is not a critical point.

We can simulate the above construction on the unit interval in order to obtain a continuous function on $[0,1]$ that has no critical point.
For a given $\sigma\in 2^{<\Nb}$, let $l_{\sigma}=0.\sigma=\sum_{i<|\sigma|}2^{-\sigma(i)}, r_{\sigma}=l_{\sigma}+2^{-|\sigma|}$, and $I_{\sigma}=[l_{\sigma},r_{\sigma}]\subseteq [0,1]$.
Now, for each $\sigma\in T^{\circ}$, let $g_{\tilde\sigma}:I_{\tilde\sigma}\to [0,3]$ be a piecewise linear function such that $g_{\tilde\sigma}(l_{\tilde\sigma})=g_{\tilde\sigma}(r_{\tilde\sigma})=3$ and $g_{\tilde\sigma}((l_{\tilde\sigma}+r_{\tilde\sigma})/2)=2-\sum_{i\in A_{\sigma}}2^{-2i}$,
and, for each $\tau \in S$, define $g_{\tau}:I_{\tau}\to [0,3]$ as $g_\tau(x)=3$.
Put $g=\bigcup_{\tau\in \tilde T^{\circ}\cup S}g_{\tau}$.
Then, $g$ is a continuous function on $[0,1]$ by Lemma~\ref{lem-code-for-conti1} and Lemma~\ref{lem-code-for-conti2}: for this, observe that for each $\tau\in \tilde T^{\circ}\cup S$, $g_{\tau}$ can be extended to an open subset of $[0,1]$, and thus $g$ can be decomposed into piecewise linear functions on open subsets of $[0,1]$.  Indeed, if $l_{\tau},r_{\tau}\notin \{0,1\}$, then one can effectively find $\sigma_{0},\sigma_{1}\in \tilde T^{\circ}\cup S$ such that $r_{\sigma_{0}}=l_{\tau}$ and $r_{\tau}=l_{\sigma_{1}}$.  Then $g$ is still piecewise linear on an interval $(l_{\sigma_{0}},r_{\sigma_{1}})$.
One can check that $g$ has no critical point, as we have seen above.
If $x\in I_{\tau}$ for some $\tau\in S$, then take any $\sigma\in T^{\circ}$, and observe that $y=(l_{\tilde\sigma}+r_{\tilde\sigma})/2$ witnesses that $x$ is not a critical point.
If $x\in I_{\tilde \sigma}$ for some $\sigma\in T^{\circ}$ then take $\sigma'\in T^{\circ}$ as in the Cantor space case.  Then $y=(l_{\tilde\sigma'}+r_{\tilde\sigma'})/2$ witnesses that $x$ is not a critical point.
\end{proof}

Next we consider versions of the variational principle provable in $\aca$. Recall from Theorem~\ref{thm-ContinuousCritPt} and Theorem~\ref{thm-CritPtsExist} that $\aca$ is able to prove the $\fvp$ when either $f$ is continuous or $\spc X$ is compact.  Let us see that each of these cases implies $\aca$, beginning with the former.

\begin{Proposition}\label{prop-FVPContinuousReversal}
The $\fvp$ for continuous, bounded $\fnc$ on the Baire space implies $\aca$ over $\rca$.
\end{Proposition}

\begin{proof}
As is well-known, $\aca$ is equivalent to the statement ``for every injection $h \colon \Nb \imp \Nb$, the range of $h$ exists as a set.'' (see~\cite[Lemma~III.1.3]{SimpsonSOSOA}).

Let $h \colon \Nb \imp \Nb$ be an injection.  For the purposes of this proof, we view the natural numbers as coding the finite sets, with $0$ coding $\varnothing$.  For each $n \in \Nb$ and finite set $D$, let $v_n(D)$ denote the number of $a \in D$ with $h(a) < n$: $v_n(D) = |\{a \in D : h(a) < n\}|$.  The fact that $h$ is an injection implies that $v_n(D) \leq n$.

Define a continuous function $g \colon \Nb^{\Nb} \imp \Rb_{\geq 0}$ by
\begin{align*}
g(x) = \sum_{n \in \Nb} 2^{-2n - 1 + v_n(x(2^{n+1}))} \leq 1.
\end{align*}
We may use Lemma~\ref{lem-code-for-conti1} to show that $g$ can indeed be coded as a continuous function in $\rca$, as it is straightforward to produce the sequence $(\la \sigma^\smf 0^\Nb, g(\sigma^\smf 0^\Nb) \ra : \sigma \in \Nb^{<\Nb})$ and, for all $n \in \Nb$ and $\sigma, \tau \in \Nb^{<\Nb}$, to check that
\begin{align*}
d(\sigma^\smf 0^\Nb, \tau^\smf 0^\Nb) < 2^{-2^{n+1}} \Imp |g(\sigma^\smf 0^\Nb) - g(\tau^\smf 0^\Nb)| < 2^{-n}.
\end{align*}
To see the above implication, observe that if $d(x,y) < 2^{-2^{n+1}}$, then $x \rest (2^{n+1} +1) = y \rest (2^{n+1}+1)$.  This means that the terms $2^{-2k - 1 + v_k(x(2^{k+1}))}$ and $2^{-2k - 1 + v_k(y(2^{k+1}))}$ agree for $k \leq n$ and therefore that $|g(x) - g(y)| < \sum_{k = n+1}^\infty 2^{-k-1} = 2^{-n-1} < 2^{-n}$.

We see that $g(x) \leq 1$ for all $x \in \Nb^{\Nb}$, thus the function $f \colon \Nb^{\Nb} \imp \Rb_{\geq 0}$ given by $f(x) = 1-g(x)$ is a continuous potential on the Baire space that is bounded by $1$.

By the $\fvp$ on the Baire space, let $x_*$ be a critical point of $f$.  For each $n$, let $D_n$ denote the set coded by $x_*(2^{n+1})$.  We claim that $D_n$ must contain every $a$ for which $h(a) < n$.  Suppose not, and let $a$ be such that $h(a) < n$ but $a \notin D_n$.  Let $y \in \Nb^\Nb$ be such that $y(2^{n+1})$ is a code for $D_n \cup \{a\}$ and $y(m) = x_*(m)$ for all $m \neq 2^{n+1}$.  Then
\begin{align*}
d(x_*, y) = 2^{-2^{n+1}} \leq 2^{-2n - 1} \leq 2^{-2n -1 + v_n(D_n)} = f(x_*) - f(y),
\end{align*}
contradicting that $x_*$ is a critical point.  Thus $D_n$ contains every $a$ for which $h(a) < n$.  We may then extract the range of $h$ from $x_*$ by taking $\ran(h) = \{n : (\exists a \in D_{n+1})(h(a) = n)\}$.
\end{proof}

\begin{Proposition}\label{PropCompactFVP}
The $\fvp$ for honestly-coded $\fnc$ on $[0,1]$ implies $\aca$ over $\rca$.
\end{Proposition}

\begin{proof}
We work in $\rca$ and prove the contrapositive. By Theorem \ref{thm-SeqCompact}, if $\aca$ fails, then there is a strictly increasing sequence $\vec c = (c_n : n \in \Nb)$ of rationals in $[0,1]$ with no supremum.

To define $\fnc$, enumerate a code $\Phi$ by enumerating $(u,v)  \vcode \Phi q$ if $q \leq u$ or if $q \leq 2$ and there is an $n$ such that $v < c_n$.  One readily checks that $\Phi$ is a code for the potential $\fnc \colon [0,1] \imp \Rb_{\geq 0}$ given by
\begin{align*}
\fnc(x) =
\begin{cases}
2 & \text{if $\exists n(x < c_n)$}\\
x & \text{otherwise}.
\end{cases}
\end{align*}
To ensure that $\Phi$ is honest, additionally enumerate $(u,v) \vcode \Phi q$ whenever there is an $n$ such that $q \leq c_n < v$.  To see that the resulting code is honest, consider an open interval $(u,v)$, and suppose that $(\forall x \in (u,v))(\fnc(x) \geq q)$.  Note that $\fnc$ is bounded above by $2$, so $q \leq 2$.  If there is an $n$ such that $v < c_n$, then $(u,v) \vcode \Phi q$.  If $c_n < u$ for all $n$, then $\fnc(x) = x$ on $(u,v)$.  Thus $q \leq u$, so $(u,v) \vcode \Phi q$.  Finally, suppose that there is an $n$ such that $c_n \in (u,v)$, but there is no $n$ such that $v < c_n$.  As $v$ is not the supremum of $(c_n : n \in \Nb)$ (because we assumed that there is no such supremum), there is a $y \in (u,v)$ such that $\forall n(c_n < y)$.  Then $\fnc(y) = y$, which implies that $q \leq y < v$.  By a similar argument, it cannot be that $\forall n(c_n < q)$ because if this were true, then there would be a $z \in (u,v)$ with $z < q$ such that $\forall n(c_n < z)$.  As also $\fnc(z) = z$, this contradicts the assumption that $(\forall x \in (u,v))(\fnc(x) \geq q)$.  Thus it must be that $q \leq c_n < v$ for some $n$, which implies that $(u,v) \vcode \Phi q$.  Hence $\Phi$ is honest.

We claim that $\fnc $ has no critical point.  Assume towards a contradiction that $c_\ast$ is a critical point of $\fnc$.  We show that $c_\ast $ is the supremum of $ \vec c$. Indeed, it is readily checked that if $c_\ast < c_n$ for some $n$, then
\[\fnc (1) = 1 \leq 2 -d(c_\ast, 1) = \fnc(c_\ast) - d(c_\ast, 1) ,\]
and hence $c_\ast$ cannot be a critical point. It follows that $c_\ast \geq c_n$ for all $n$. Now, if $c' < c_\ast$ were also an upper bound of $\vec c$, then we would have $\fnc (c') = c'$, so that
\[\fnc (c') = c' = c_\ast - (c_\ast - c') = \fnc (c_\ast) - d(c', c_\ast),\]
and thus $c_\ast$ cannot be a critical point. Hence $c_\ast$ is the supremum of $\vec c$, contradicting our initial assumption.
\end{proof}

Finally, we show that the unrestricted $\fvp$ proves $\pica$ by appealing to the following characterization of $\pica$.

\begin{Lemma}[{\cite[Lemma~VI.1.1]{SimpsonSOSOA}}]
The following are equivalent over $\rca$:

\begin{enumerate}

\item $\pica$ 

\item for any sequence $(T_i)_{i \in \Nb}$ of subtrees of $\Nb^{<\Nb}$, there is a set $X$ such that for all $i\in \Nb$, $i \in X $ if and only if $T_i$ has an infinite path.

\end{enumerate}

\end{Lemma}

\begin{Proposition}\label{prop:critical-vs-Pi11CA-RCA}
The $\fvp$ on the Baire space implies $\pica$ over $\rca$.
\end{Proposition}

\begin{proof}
By Proposition~\ref{prop-FVPContinuousReversal}, the $\fvp$ on the Baire space implies $\aca$ over $\rca$, so we may work over $\aca$.

Recall that we assume that the pairing function $\la \cdot, \cdot \ra \colon \Nb \times \Nb \imp \Nb$ is increasing in both coordinates.  For the purposes of this proof, if $x \in \Nb^\Nb$ and $i \in \Nb$, then $(x)_i \in \Nb^\Nb$ is the function defined by $(x)_i(n) = x(\la i, n \ra)$.  Similarly, if $\sigma \in \Nb^{<\Nb}$ and $i \in \Nb$, then $(\sigma)_i \in \Nb^{<\Nb}$ is the longest sequence such that $(\forall n < |(\sigma)_i|)[\la i, n \ra \in \dom \sigma \andd (\sigma)_i(n) = \sigma(\la i,n \ra)]$.

Let $(T_i)_{i \in \Nb}$ be a sequence of subtrees of $\Nb^{<\Nb}$.  We first define a code for the lower semi-continuous potential $\fnc \colon \Nb^\Nb \imp [0,1]$ given by
\begin{align*}
\fnc(x) = \sum_{i=0}^\infty \{2^{-i} : (x)_{i} \notin [T_i]\}.
\end{align*}
To do this, define $\ball{\sigma}{r} \vcode \Phi q$ if there is a $\tau \in \Nb^{<\Nb}$ such that $\ball{\sigma}{r} \ballsub \ball{\tau}{2^{-|\tau|}}$ and $q \leq \sum_{i=0}^\infty \{2^{- i } : (\tau)_{i } \notin T_i\}$.  For a given $x \in \Nb^\Nb$, one readily checks that $v = \sum _{i=0}^\infty \{2^{- i } : (x)_{ i } \notin [T_i]\}$ (which $\aca$ proves exists) is indeed the supremum of
\[\{q \in \Qb : (\exists \la \sigma,r \ra \in \Nb^{<\Nb} \times \Qb_{>0})(\ball \sigma r \vcode \Phi q \andd d(x,\sigma) < r)\}.\]
This shows that $\Phi$ correctly codes the desired potential $\fnc$ and that $\fnc$ is provably total in $\aca$.

By the $\fvp$, let $x_*$ be a critical point of $\fnc$.  Let $X = \{{i} : (x_*)_{i } \in [T_i]\}$.  We show that $(\forall i \geq 0 )  (i \in X \biimp \text{$T_i$ has a path})$.  Clearly, if $i \in X$, then $T_i$ has a path.  Suppose for a contradiction that there is a $j$ such that $T_j$ has a path, but $j \notin X$.  Let $h \in [T_j]$, and let $y \in \Nb^\Nb$ be such that for all $i , n \in \Nb$,
\begin{align*}
y(\la i,n \ra) =
\begin{cases}
h(n) & \text{if $i = j  $}\\
x_*(\la i,n \ra) & \text{otherwise},
\end{cases}
\end{align*}
so that $(y)_{j } = h$ and $\forall i [i \neq j   \imp (y)_i = (x_*)_i]$.  Then
\begin{align*}
d(x_*, y) \leq 2^{-\la j  , 0 \ra} \leq 2^{- j } = \fnc(x_*) - \fnc(y).
\end{align*}
This contradicts that $x_*$ is a critical point of $\fnc$, which completes the proof.
\end{proof}

\section{The localized variational principle}\label{secLVP}

In this section we compare the strength of the free and localized variational principles in different contexts.
First, we show that, in contrast to the $\fvp$, the $\lvp$ is not affected by the boundedness of $\fnc$.
We also show that the two principles are equivalent when $\fnc$ is not assumed to be continuous.
For the latter, it is clear that the $\lvp$ implies the $\fvp$, so we focus on the other direction.
The following two lemmas aid the proof of the above facts.

\begin{Lemma}[$\rca$]\label{lem-lscOpen}
Let $\spc X$ be a complete separable metric space, let $f \colon \spc X \imp \overline{\Rb}_{\geq 0}$ be lower semi-continuous, and let $\MC C \subseteq \spc X$ be closed.  Then there is a lower semi-continuous function $g \colon \spc X \imp \overline{\Rb}_{\geq 0}$ such that
\begin{align*}
g(x) =
\begin{cases}
0 & \text{if $x \in \MC C$}\\
f(x) & \text{if $x \not \in  \MC C$.}
\end{cases}
\end{align*}
\end{Lemma}

\begin{proof} 
Let $\Psi$ be a code for $f$, and let $U$ be a code for the complement of $\MC C$.  Enumerate a code $\Phi$ for $g$ by enumerating $\ball{a}{r} \vcode \Phi q$ if either
\begin{itemize}
\item $q \leq 0$ or
\item $\ball{a}{r} \vcode \Psi q$ and there is an open ball $\ball{b}{s}$ enumerated in $U$ with $\ball{a}{r} \ballsub \ball{b}{s}$.
\end{itemize}
If $x \in \MC C$ then there are balls $\ball{a}{r}$ with $x \in \ball{a}{r}$ and $\ball{a}{r} \vcode \Phi 0$, but there are no balls $\ball{a}{r}$ and $q > 0$ such that $x \in \ball{a}{r}$ and $\ball{a}{r} \vcode \Phi q$.  Hence $g(x) = 0$.
If $x \notin \MC C$, given $q \in \Qb_{>0}$, it is clear from the definition of $\Phi$ that if there is a ball $\ball{a}{r}$ with $x \in \ball{a}{r}$ and $\ball{a}{r} \vcode \Phi q$ then also $\ball{a}{r} \vcode \Psi q$.
Conversely, suppose that there is a ball $\ball{a}{r}$ with $x \in \ball{a}{r}$ and $\ball{a}{r} \vcode \Psi q$.
Since $x \notin \MC C$, $x \in \ball{b}{s}$ for some ball $\ball{b}{s}$ enumerated in $U$, and hence there are $a', r'$ such that $x \in \ball{a'}{r'}\ballsub \ball ar, \ball{b}{s} $. But then by condition ({\sc lsc}1) of Definition~\ref{def-lsc}, $\ball{a'}{r'} \vcode \Psi q$, so that also $\ball{a'}{r'} \vcode \Phi q$.  Since $q$ was arbitrary we conclude that $g(x) = f(x)$ given that both are the supremum of such $q$ in $\overline \Rb$.
\end{proof}

\begin{Lemma}[$\rca$]\label{lem-lscSumMax}
Let $\spc X$ be a complete separable metric space, and let $f ,g \colon \spc X \imp  \overline{\Rb}_{\geq 0}$ be lower semi-continuous.  Then $(f+g),\max\{f,g\},\min\{f,g\} \colon \spc X \imp \overline{\Rb}_{\geq 0}$ (defined in the usual way) are lower semi-continuous.
\end{Lemma}

\begin{proof}
Let $\Psi_0$ be a code for $f$ and let $\Psi_1$ be a code for $g$.  Enumerate a code for $f+g$ by enumerating $\ball{a}{r} \vcode{\Phi} q$ if there are $q_0,q_1$ such that $\ball{a}{r} \vcode{\Psi_0} q_0$, $\ball{a}{r} \vcode{\Psi_1} q_1$, and $q \leq q_0 + q_1$.
Enumerate a code $\Phi$ for $\max\{f,g\}$ by enumerating $\ball{a}{r} \vcode \Phi q$ if either $\ball{a}{r} \vcode {\Psi_0} q$ or $\ball{a}{r} \vcode {\Psi_1} q$.
Enumerate a code $\Phi$ for $\min\{f,g\}$ by enumerating $\ball{a}{r} \vcode \Phi q$ if both $\ball{a}{r} \vcode {\Psi_0} q$ and $\ball{a}{r} \vcode {\Psi_1} q$.
\end{proof}

\begin{Lemma}[$\rca$]\label{lemmBoundedLVP}
For every complete separable metric space $\spc X$, the $\lvp$ for bounded potentials on $\spc X$ implies the $\lvp$ on $\spc X$, and the $\lvp$ for bounded continuous potentials on $\spc X$ implies the $\lvp$ for continuous potentials on $\spc X$.
\end{Lemma}

\begin{proof}
Given a complete separable metric space $\spc X$, a potential $f \colon \spc X \to \overline{\Rb}_{\geq 0}$, and an $x_0 \in \lscsupp (f)$, define $\tilde{f} \colon \spc X \to \Rb_{\geq 0}$ by setting $\tilde{f}(x) = \min \{f(x), f(x_0)\}$.  Then $\tilde{f}$ is bounded, is a potential by Lemma~\ref{lem-lscSumMax}, and is continuous if $f$ is continuous.  Suppose that $\varepsilon > 0$ and that $x_* \in \spc X$ is an $\varepsilon$-critical point of $\tilde f$ with $\varepsilon d(x_0, x_*) \leq \tilde{f}(x_0) - \tilde{f}(x_*)$.  Then either $\tilde{f}(x_*) = f(x_*)$ or $\tilde{f}(x_*) = f(x_0)$.  However, if $\tilde{f}(x_*) = f(x_0)$, then $\varepsilon d(x_0, x_*) \leq \tilde{f}(x_0) - \tilde{f}(x_*) = f(x_0) - f(x_0) = 0$.  Thus $d(x_0, x_*) = 0$ and $x_* = x_0$, in which case again $\tilde{f}(x_*) = f(x_*)$.  Thus in either case $\tilde{f}(x_*) = f(x_*)$.  Therefore $\varepsilon d(x_0, x_*) \leq \tilde{f}(x_0) - \tilde{f}(x_*) = f(x_0) - f(x_*)$.  Furthermore, if $x \in \spc X$ and $\varepsilon d(x_*, x) \leq f(x_*) - f(x)$, then also $\varepsilon d(x_*, x) \leq \tilde{f}(x_*) - \tilde{f}(x)$, so $x = x_*$ because $x_*$ is an $\varepsilon$-critical point of $\tilde f$.  Thus $x_*$ is also an $\varepsilon$-critical point of $f$ with $\varepsilon d(x_0, x_*) \leq f(x_0) - f(x_*)$.
\end{proof}

\begin{Lemma}[$\rca$]\label{lemmFVPtoLVP}
For every complete separable metric space $\spc X$, the $\fvp$ on $\spc X$ implies the $\lvp$ on $\spc X$.
\end{Lemma}

\begin{proof}
Assume the $\fvp$.  Let $f \colon \spc X \to \overline{\Rb}_{\geq 0}$ be any potential, $\varepsilon > 0$, $x_0 \in \lscsupp (f)$, and let $\MC C$ be the closed set
\[\mathcal C = \{x\in \spc X : \varepsilon d(x,x_0) \leq f(x_0) -  f(x) \}.\]
To see that $\mathcal C$ is closed, rewrite $\mathcal C$ as $\mathcal C = \{x\in \spc X : f(x) + \varepsilon d(x,x_0) \leq f(x_0) \}$, and observe that the function $g(x) = f(x) + \varepsilon d(x,x_0)$ is lower semi-continuous by Proposition~\ref{prop-LSCchar} and Lemma~\ref{lem-lscSumMax}.  As in the proof of Proposition~\ref{prop-LSCandGraph}, $g(x) \leq f(x_0)$ is a $\Pi^0_1$-definable property of $x$, so $\mathcal C$ is closed.

By Lemma~\ref{lem-lscOpen} and Lemma~\ref{lem-lscSumMax}, the function $\tilde f \colon \spc X \to \overline{\Rb}_{\geq 0}$ given by 
\begin{align*}
\tilde{f}(x) =
\begin{cases}
f(x) & \text{if $x \in \MC C$}\\
\max\{f(x), \varepsilon d(x, x_0) + f(x_0)\} & \text{if $x \notin \MC C$}
\end{cases}
\end{align*}
is lower semi-continuous.  Let $x_\ast$ be an $\varepsilon$-critical point of $\tilde f$.  It must be that $x_\ast \in \MC C$, for otherwise we would have that $x_\ast \neq x_0$ (as $x_0 \in \MC C$), but $\tilde{f}(x_\ast) \geq \varepsilon d(x_\ast, x_0) + f(x_0) = \varepsilon d(x_\ast, x_0) + \tilde{f}(x_0)$, meaning that $x_0$ witnesses that $x_\ast$ is not an $\varepsilon$-critical point of $\tilde f$.

We have that $f(x_\ast) \leq f(x_0) - \varepsilon d(x_\ast,x_0)$ because $x_\ast \in \mathcal C$.  Moreover, if $x \in \spc X$ and $f(x) \leq f(x_\ast) - \varepsilon d(x_\ast,x)$, then
\[f(x) \leq f(x_\ast) - \varepsilon d(x_\ast,x) \leq f(x_0) - \varepsilon d(x_\ast,x_0) - \varepsilon d(x_\ast,x) \leq f(x_0) - \varepsilon d(x,x_0)\]
(where the last inequality is by the triangle inequality), so $x \in \mathcal C$ and therefore $\tilde{f}(x) = f(x)$.  Thus for an $x \in \spc X$ with $f(x) \leq f(x_\ast) - \varepsilon d(x_\ast,x)$,
\[\tilde f(x) = f(x) \leq f(x_\ast) - \varepsilon d(x_\ast,x) = \tilde f(x_\ast) - \varepsilon d(x_\ast,x),\]
and therefore $x=x_\ast$ because $x_\ast$ is an $\varepsilon$-critical point of $\tilde f$.  So $x_\ast$ is also an $\varepsilon$-critical point of $f$.
We conclude that $x_\ast$ is an $\varepsilon$-critical point of $f$ and that $f(x_\ast) \leq f(x_0) - \varepsilon d(x_\ast, x_0)$.
\end{proof}

Thus the $\fvp$ and the $\lvp$ for arbitrary $f$ are equivalent.
In contrast, the $\lvp$ for continuous functions suffices to prove the full $\fvp$.
To make this precise, we introduce the notion of {\em pseudo-fibrations.}
Below, by a {\em closed isometry} we mean a continuous function $f\colon \spc X \to \spc Y$ such that $ d_\spc X (x_0,x_1) =  d_\spc Y(f(x_0),f(x_1) ) $ for all $x_0 , x_1 \in \spc X$ and such that $f[\spc X]$ is closed.  We check that if $f$ is a closed isometry and $\mathcal C \subseteq \spc X$ is closed, then $\rca$ proves that $f[\mathcal C]$ is closed as well.

\begin{Lemma}[$\rca$]\label{lem-IsoImage}
Let $\spc X$ be a complete separable metric space, let $\MC C \subseteq \spc X$ be closed, and let $f \colon \spc X \to \spc Y$ be a closed isometry.  Then $f[\MC C] \subseteq \spc Y$ is closed.
\end{Lemma}

\begin{proof}
Let $U$ be a code for the complement of $\MC C \subseteq \spc X$, and let $V_0$ be a code for the complement of $f[\spc X] \subseteq \spc Y$.  Enumerate a code $V$ for an open $\MC V \subseteq \spc Y$ by enumerating the ball $\ball{b}{s}$ in $V$ if either
\begin{itemize}
\item $\ball{b}{s}$ is enumerated in $V_0$ or
\item there is a ball $\ball{a}{r}$ enumerated in $U$ such that $d_{\spc{Y}}(f(a),b) + s < r$.
\end{itemize}
We show that a $y \in \spc Y$ is in $\MC V$ if and only if there is no $x \in \MC C$ such that $f(x) = y$.  This shows that $\MC V$ is the complement of $f[\MC C]$ and hence that $f[\MC C]$ is closed.

For $y \in \spc Y$, we know that there is no $x \in \spc X$ such that $y = f(x)$ if and only if there is a ball $\ball{b}{s}$ containing $y$ that is enumerated in $V_0$.  Thus it suffices to assume that there is an $x \in \spc X$ such that $y = f(x)$ and show that $y \in \MC V$ if and only if $x \notin \MC C$.  Note that $x$ is unique because $f$ is an isometry.

First suppose that $x \notin \MC C$.  Then $x \in \ball{a}{r}$ for some ball $\ball{a}{r}$ enumerated in $U$.  We have that $d_{\spc Y}(f(a), y) = d_{\spc X}(a,x) < r$ because $y = f(x)$ and $f$ is an isometry.  Thus by choosing $b$ sufficiently close to $y$, we see that there is a ball $\ball{b}{s}$ containing $y$ with $d_{\spc{Y}}(f(a), b) + s < r$.  Thus $y \in \MC V$.

Conversely, suppose that $y \in \MC V$.  Then there are balls $\ball{b}{s} \subseteq \spc Y$ and $\ball{a}{r} \subseteq \spc X$ such that $y \in \ball{b}{s}$, $\ball{a}{r}$ is enumerated in $U$, and $d_{\spc{Y}}(f(a), b) + s < r$.  Therefore $d_{\spc Y}(f(a), f(x)) < r$ because $f(x) = y \in \ball{b}{s}$.  Thus $d_{\spc X}(a, x) < r$ because $f$ is an isometry.  Therefore $x \in \ball{a}{r}$, so $x \notin \MC C$.
\end{proof}

\begin{Definition}\label{defPseudofib}
Let $\spc X, \spc Y$ be complete separable metric spaces.  Say that a space $\spc Z$ is an \emph{$\spc X$-pseudofibration} of $\spc Y$ if there are a closed isometry $\iota : \spc X \times \spc Y \to \spc Z$ and a continuous function $\pi \colon \spc Z \to \spc Y$ such that
\begin{enumerate}

\item \label{itPseudofib1} for all $z,z' \in \spc Z$, $d_\spc Y (\pi(z),\pi (z')) \leq d_\spc Z(z,z')$ and

\item \label{itPseudofib2} for all $\la x,y\ra  \in \spc X \times \spc Y$, $\pi  \iota (  x,y  ) = y$.
\end{enumerate}
\end{Definition}

Of course the typical example of an $\spc X$-pseudofibration of $\spc Y$ is $\spc X \times \spc Y$, but later we will see that there are others.
Pseudofibrations are useful due to the following.

\begin{Lemma}[$\rca$]\label{lemmLVPtoFVP}\
\begin{enumerate}

\item

If $\spc X$ and $\spc Z$ are complete separable metric spaces such that $\spc Z$ is an $\spc X$-pseudofibration of $\Rb_{\geq 0}$, then the $\lvp$ for continuous potentials on $\spc Z$ implies the $\fvp$ on $\spc X$.

\item

If $\spc X$ and $\spc Z$ are complete separable metric spaces such that $\spc Z$ is an $\spc X$-pseudofibration of $[0,1]$, then the $\lvp$ for continuous potentials on $\spc Z$ implies the $1$-$\fvp$ on $\spc X$.

\end{enumerate}
\end{Lemma}

\begin{proof} 
We give a uniform proof for the two claims.
Note that by Lemma \ref{lemEpsRest}.\ref{itEpsUnbounded} we can also assume for the conclusion of the first claim that $\varepsilon = 1$.
Let $\spc X$ be a complete separable metric space, let $\spc Y$ be either $\Rb_{\geq 0}$ for the first claim or $[0,1]$ for the second, and let $\spc Z$ be an $\spc X$-pseudofibration of $\spc Y$ with associated functions $\iota \colon \spc{X} \times \spc Y \to \spc{Z}$ and $\pi \colon \spc Z \to \spc Y$.  For the first claim, let $\fnc \colon \spc X \to \overline{\Rb}_{\geq 0}$ be lower semi-continuous; and for the second claim, let $\fnc \colon \spc X \to [0, 1]$ be lower semi-continuous.  By Proposition~\ref{prop-LSCandGraph}, the set
\begin{align*}
\Delta = \{\la x, y \ra \in \spc X \times \spc Y : \fnc (x) \leq y\}
\end{align*}
is closed, hence so is $\Gamma = \iota [\Delta] \subseteq \spc Z$ by Lemma~\ref{lem-IsoImage} because $\iota$ is a closed isometry.  By Urysohn's lemma, which is provable in $\rca$ by~\cite[Lemma~II.7.3]{SimpsonSOSOA}, there is a continuous function $g \colon \spc Z \to [0,1]$ such that $g(z) = 0$ if and only if $z\in \Gamma$.  Define a continuous function $\tilde\fnc \colon \spc Z \to \Rb_{\geq 0}$ by
\begin{align*}
\tilde\fnc(z) = \pi(z) + g(z).
\end{align*}
Every $z \in \Gamma$ is of the form $z = \iota( x, y )$ for some $\la x, y \ra \in \spc X \times \spc Y$.  For such a $z$, we then have that
\begin{align*}
\tilde\fnc(z) = \pi(z) + g(z) = \pi(z) = \pi \iota(  x, y ) = y,
\end{align*}
where the second equality is because $z \in \Gamma$ (and therefore $g(z) = 0$), and the last equality is by Definition~\ref{defPseudofib}.\ref{itPseudofib2}.

Now, fix any $x_0 \in \lscsupp (f)$, let $y_0 = \fnc(x_0)$, and let $z_0 = \iota( x_0, y_0 )$.  By the $\lvp$ for continuous potentials on $\spc Z$, let $z_* \in \spc Z$ be a critical point of $\tilde\fnc$ that satisfies $\tilde\fnc(z_*) \leq \tilde\fnc(z_0) - d_{\spc Z}(z_0, z_*)$.  We first claim that $z_* \in \Gamma$.  To see this, observe that $z_0 \in \Gamma$ by definition and therefore $\tilde\fnc(z_0) = \pi(z_0) = y_0$.  We then have that
\begin{align*}
\pi(z_*) + g(z_*) = \tilde\fnc(z_*) &\leq \tilde\fnc(z_0) - d_{\spc Z}(z_0, z_*)\\
&= \pi(z_0) - d_{\spc Z}(z_0, z_*)\\
&\leq \pi(z_0) - d_{\spc Y}(\pi(z_0), \pi(z_*)) & &\text{by Definition~\ref{defPseudofib}.\ref{itPseudofib1}}\\
&= \pi(z_0) - |\pi(z_0) - \pi(z_*)|\\
&\leq \pi(z_*).
\end{align*}
Therefore we must have that $g(z_*) = 0$, which means that $z_* \in \Gamma$.

As $z_* \in \Gamma$, it must be that $z_* = \iota( x_*, y_* )$ for some $\la x_*, y_* \ra \in \Delta$.  Thus $x_* \in \spc X$ and $y_* \geq \fnc(x_*)$.  In fact, it must be that $y_* = \fnc(x_*)$.  To see this, suppose for a contradiction that $y_* > \fnc(x_*)$, and let $z = \iota(  x_*, \fnc(x_*) )$.  We then have that
\begin{align*}
\tilde\fnc(z_*) - \tilde\fnc(z) = y_* - \fnc(x_*) = d_{\spc Z}(z, z_*),
\end{align*}
where the first equality is because $z_*$ and $z$ are in $\Gamma$, and the second equality is because $\iota$ is an isometry.  Thus $z$ witnesses that $z_*$ is not a critical point of $\tilde\fnc$, which is a contradiction.  Therefore $y_* = \fnc(x_*)$, so $z_* = \iota( x_*, \fnc(x_*) )$.

We now show that $x_*$ is a critical point of the original $\fnc$, which completes the proof.  Let $x \in \spc X$, and suppose that $\fnc(x_*) \geq \fnc(x) + d_{\spc X}(x, x_*)$.  Let $z = \iota( x, \fnc(x) )$.  Then
\begin{align*}
d_{\spc Z}(z_*, z) = d_{\spc X \times \spc Y}(\la x_*, \fnc(x_*) \ra, \la x, \fnc(x) \ra) = \fnc(x_*) - \fnc(x),
\end{align*}
where the first equality is because $\iota$ is an isometry, and the second equality is because $\fnc(x_*) - \fnc(x) \geq d_{\spc X}(x, x_*)$ and we use the $\max$ norm on $\spc X \times \spc Y$.  As $z_*$ and $z$ are in $\Gamma$, we have that
\begin{align*}
\tilde\fnc(z_*) - \tilde\fnc(z) = \fnc(x_*) - \fnc(x) = d_{\spc Z}(z_*, z).
\end{align*}
Therefore $z = z_*$ because $z_*$ is a critical point of $\tilde\fnc$, and since $\iota$ is injective, $x = x_*$.  We have shown that if $x \in \spc X$ satisfies $\fnc(x_*) \geq \fnc(x) + d_{\spc X}(x, x_*)$, then $x = x_*$.  Thus $x_*$ is a critical point of $\fnc$.
\end{proof}

Of course it follows from Lemma \ref{lemmLVPtoFVP} that the $\lvp$ for continuous functions on $\spc X \times \Rb_{\geq 0}$ implies the $\fvp$ on $\spc X$.
However, the notion of pseudofibrations will allow us to replace $\spc X \times \Rb_{\geq 0}$ by a more familiar space.
In what follows, we will mainly deal with $\mathcal C \big ( [0,1] \big ) $.
For given $h\in \mathcal C \big ( [0,1] \big ) $ and a closed interval $I\subseteq [0,1]$, we write $\|h\|_{I}=\sup_{t\in I}|h(t)|$.
By the definition (coding) of $\mathcal C \big ( [0,1] \big ) $, statements of the form $\|h\|_{I}\le r$ are always $\Pi^{0}_{1}$.

\begin{Lemma}[$\rca$]\label{lemmCZOPF}
Let $\spc X$ be a complete separable metric space, and let $\spc Y = [a,b] \cap \Rb$, where $a < b$ are elements of $\overline\Rb$. If $\spc X$ embeds by a closed isometry into $\mathcal C \big ( [0,1] \big ) $, then $\mathcal C \big ( [0,1] \big ) $ is an $\spc X$-pseudofibration of $\spc Y$.
\end{Lemma}

\begin{proof}
Let $\spc Z = \mathcal C \big ( [0,1] \big )$ and let $f\colon \spc X \to \spc Z$ be a closed isometry.
Let us write $f_x$ instead of $f(x)$.
Define $\iota \colon \spc X \times \spc Y \to \spc Z$ by
\[
\big ( \iota (x,y) \big ) (t)=
\begin{cases}
2t f_x(0)+(1-2t) y & \text{if $t< \nicefrac 12$,}\\
f_x   ( 2 t - 1  ) & \text{if $t\geq \nicefrac 12$,}
\end{cases}
\]
and define $\pi \colon \spc Z \to \spc Y$ by
\[\pi(g) =
\begin{cases}
a & \text{if $g(0)<a$,}\\
b & \text{if $g(0)>b $,}\\
g(0) & \text{otherwise.}\\
\end{cases}
\]
Intuitively, $\iota (x,y)$ represents $y$ by its value on $0$ and represents $f(x)$ by its values on $[\nicefrac 12,1]$.
It is easy to check that a code for $\pi$ can be constructed using Lemma~\ref{lem-code-for-conti1}.
We may also use Lemma~\ref{lem-code-for-conti1} to construct a code for $\iota$.
To do this, let $X$ and $Y$ be the dense sets of $\spc X$ and $\spc Y$, respectively.  It is straightforward to produce the sequence $(\la \la a,q \ra, \iota(a,q) \ra : \la a,q \ra \in X \times Y)$ and observe that
\begin{align}
\label{IotaIso} d_{\spc X \times \spc Y}(\la a_0 ,q_0 \ra, \la a_1, q_1 \ra) = \max(d_{\spc X}(a_0, a_1), d_{\spc Y}(q_0, q_1)) = d_{\spc Z}(\iota(a_0, q_0), \iota(a_1, q_1)).
\end{align}
Thus the hypotheses of Lemma~\ref{lem-code-for-conti1} are satisfied, so there exists a code for $\iota$.
Equation~\eqref{IotaIso} implies that $\iota$ is an isometry. To see that $\iota$ is closed, consider an $h\in\mathcal C\big ( [0,1] \big )$, put $h_{\flat}(t)=(2t)h(1/2)-(1-2t)h(0)$, and put $h_{\sharp}(t)=h(\nicefrac t2 + \nicefrac 12)$.
Then $h \in \iota[\spc X \times \spc Y]$ if and only if $\|h-h_{\flat}\|_{[0,1/2]}=0$ (i.e., $h$ is linear on $[0,1/2]$), $h(0)\in\spc Y$, and $h_\sharp \in f[\spc X]$, which is a $\Pi^{0}_{1}$ condition because $f[\spc X]$ is closed by assumption.
Thus, $\iota[\spc X \times \spc Y]$ is closed by Lemma~\ref{lem-code-for-open}.  It is then easy to check that $\iota$ and $\pi$ make $\spc Z$ an $\spc X$-pseudofibration of $\spc Y$.
\end{proof}

In order to apply Lemma \ref{lemmCZOPF}, we need to show that many of the spaces we are interested in isometrically embed into $\mathcal C \big ( [0,1] \big )$. Let us begin with the unit interval.

\begin{Lemma}[$\rca$]\label{lemUnitEmbeds}
There exists a closed isometry $f\colon [0,1] \to \mathcal C \big ( [0,1] \big )$.
\end{Lemma}

\begin{proof}
Let $\spc X = [0,1]$ and $\spc Y = \mathcal C \big ( [0,1] \big )$.
For $r\in[0,1]$, write $\mathcal I_{r}$ for the constant function with value $r$.
Define $f \colon \spc X \to \spc Y$ by $f(x)= \mathcal I_{x}$.
It is obvious that $f$ is an isometry and thus a code for $f$ exists by arguing in the style of the proof of Lemma~\ref{lemmCZOPF} and appealing to Lemma~\ref{lem-code-for-conti1}.
We may observe that $h\in f[\spc X]$ if and only if $\|h-\mathcal I_{h(0)}\|_{[0,1]}=0$, which is a $\Pi^0_1$ condition, thus $f[\spc X]$ is closed by Lemma~\ref{lem-code-for-open}.
\end{proof}

Next, let us see that there is also a closed isometry from the Baire space into $\mathcal C \big ( [0,1] \big )$.

\begin{Definition}\label{DefBaireEmbed}
If $I = [a,b] $ is an interval with $0\leq a < b \leq 1$, define $\hat I \colon [0,1] \to \Rb$ to be the piecewise linear function with $\hat I(0) = \hat I(a) = \hat I(b) = \hat I(1) = 0$, $\hat I \big ( \frac{a+b}2 \big ) = 1$, and $\hat I$ linear elsewhere.
For $n \in \Nb$, define $I_n = [1-2^{-n},1-2^{-(n+1)}]$ and for $m,n \in \Nb$ define
\[J_{n}^m = \big [ 1-2^{-(n+1)} - 2^{-(n+m+1)} , 1-2^{-(n+1)} - 2^{-(n+m+2)} \big ].\]
Let $\spc X$ be the Baire space and $\spc Y$ be $\mathcal C \big ( [0,1] \big )$.
Then, define $F \colon \spc X \to \spc Y$ by
\[F( {x} ) = \sum_{n \in \Nb} 2^{-n} \hat J_ {{n } }^{{x}(n)}.\]
\end{Definition}

\begin{Lemma}[$\rca$]\label{lemBaireEmbeds}
The function $F$ of Definition \ref{DefBaireEmbed} is a closed isometry from the Baire space to ${\mathcal C} \big ( [ 0, 1 ] \big ) $.
\end{Lemma}

\begin{proof}
A code for $F$ exists by checking that $F$ is an isometry, arguing as in the proof of Lemma~\ref{lemmCZOPF}, and appealing to Lemma~\ref{lem-code-for-conti1}.
For $h\in {\mathcal C} \big ( [ 0, 1 ] \big ) $, $h\in F[\Nb^{\Nb}]$ if and only if for every $n\in\Nb$, $\|h\|_{I_{n}}=2^{-n}$, and for every $n,m\in\Nb$, either $\|h\|_{J_{n}^{m}}=0$ or $\|h-2^{-n}\hat J_{n}^{m}\|_{I_{n}}=0$, which is a $\Pi^{0}_{1}$ statement.
It is easy to check the ``only if'' direction, so we check the ``if'' direction. For a given $h\in {\mathcal C} \big ( [ 0, 1 ] \big )$ which satisfies the assumption, define $h_{\Nb}:\Nb\to\Nb$ by letting $h_{\Nb}(n)$ be the unique $m\in\Nb$ such that $\|h\|_{J_{n}^{m}}\neq 0$ (which can be found effectively).  Then $F(h_{\Nb})=h$.  Thus the image of $F$ is closed by Lemma~\ref{lem-code-for-open}.
\end{proof}

\begin{Proposition}[$\rca$]\label{propLVPReverse}\
\begin{enumerate}

\item \label{itLVPReverseACA} The $\lvp$ for bounded, continuous potentials on $[0,1] \times  [0,1]$ with the max metric implies $\aca$.

\item \label{itLVPReversePica} The $\lvp$ for bounded, continuous potentials on $\mathcal C \big ( [0,1] \big )$ implies $\pica$.

\end{enumerate}
\end{Proposition}

\begin{proof}
Note that in view of Lemma \ref{lemmBoundedLVP} the bounded $\lvp$ implies the full $\lvp$, so we may remove `bounded' from both items.
For the first item, we note by Proposition \ref{PropCompactFVP} that the $\fvp$ on $[0,1]$ implies $\aca$.
Moreover, by Lemma \ref{lemEpsRest}.\ref{itBallRn}, since $[0,1]$ is a closed ball in $\Rb^1$, we may replace the $\fvp$ by the $1$-$\fvp$.
Since $[0,1] \times [0,1]$ is clearly a $[0,1]$-pseudofibration of $[0,1]$, by Lemma \ref{lemmLVPtoFVP} we have that the $\lvp$ for bounded, continuous potentials on $[0,1] \times [0,1]$ implies $\aca$ as well.

For the second, let $\spc X$ be the Baire space and $\spc Z$ be ${\mathcal C} \big ( [0,1] \big )$. By Lemma \ref{lemEpsRest}.\ref{itBaireSpace} and Proposition \ref{prop:critical-vs-Pi11CA-RCA}, the $1$-$\fvp$ on $\spc X$ implies $\pica$. By Lemma \ref{lemBaireEmbeds}, $\spc X$ embeds by a closed isometry into $\spc Z$, so that by Lemma \ref{lemmCZOPF}, $\spc Z$ is an $\spc X$-pseudofibration of $[0,1]$. Hence, reasoning as above, the $\lvp$ for bounded, continuous functions on $\spc Z$ implies the $1$-$\fvp$ on $\spc X$, from which we obtain $\pica$.
\end{proof}

\begin{remark}\label{RemUnitBall}
Let $\spc B$ be the closed unit ball of $\mathcal C \big ( [0,1] \big )$. It is easy to see that the isometries used in Lemmas \ref{lemUnitEmbeds} and \ref{lemBaireEmbeds} map into $\spc B$, and the proof of Lemma \ref{lemmCZOPF} can be modified to show that $\spc B$ is an $\spc X$-pseudofibration of $[0,1]$. With this we may replace $\mathcal C \big ( [0,1] \big )$ by $\spc B$ in Proposition \ref{propLVPReverse}.\ref{itLVPReversePica}.
\end{remark}

\section{Conclusion}\label{secConc}

We have considered formalizations of Ekeland's variational principle with various natural restrictions and have shown that they are equivalent to well-known theories of second-order arithmetic. In this final section we synthesize these results to produce the definitive versions of our main results.

There are three main weakenings of the $\lvp$ we considered: the free, rather than localized, statement; the case where $\spc X$ is compact; and the case where $f$ is continuous. When the $\lvp$ is weakened in the three ways simultaneously, we obtain a statement equivalent to $\wkl$.

\begin{Theorem}
Over $\rca$, the following are equivalent:

\begin{enumerate}[label=(\alph*)]

\item \label{ItWKL} $\wkl$;

\item the $\fvp$ for continuous $\fnc$ and compact $\spc X$;

\item \label{ItCantor} the $\fvp$ for continuous $\fnc$ on the Cantor space or on $[0,1]$.

\end{enumerate}

\end{Theorem}

\begin{proof}
That \ref{ItWKL} implies the other items is Proposition \ref{propWKLtoFVP}, while \ref{ItCantor} implies \ref{ItWKL} by Proposition \ref{propFVPtoWKL}.
\end{proof}

If we choose two, but not three, of the weakenings, we instead obtain statements equivalent to $\aca$. Moreover, if we impose only one of the three weakenings, we typically obtain statements that are not provable in $\aca$.  The only exception to this is the case where $\spc X$ is compact.

\begin{Theorem}\label{TheoCritEquivACA}
Over $\rca$, the following are equivalent:

\begin{enumerate}[label=(\alph*)]

\item \label{ItACA} $\aca$;

\item\label{ItContFVP} the $\fvp$ for continuous $\fnc$;

\item \label{ItCompactFVP} the $\fvp$ for compact $\spc X$;

\item\label{ItCompLVP} the $\lvp$ for compact $\spc X$;

\item\label{ItCompContLVP} the $\lvp$ for continuous $f$ and compact $\spc X$.

\end{enumerate}
Moreover in item \ref{ItContFVP} we may take $\spc X$ to be the Baire space, in items \ref{ItCompactFVP} and \ref{ItCompLVP} we may take $\spc X = [0,1]$, and in \ref{ItCompLVP} and \ref{ItCompContLVP} we may take $\spc X = [0,1] \times [0,1]$ with the max metric.
\end{Theorem}

\begin{proof}
The equivalence between \ref{ItACA} and \ref{ItContFVP} is given by Theorem \ref{thm-ContinuousCritPt} and Proposition~\ref{prop-FVPContinuousReversal}, while the equivalence between \ref{ItACA} and \ref{ItCompactFVP} is given by Theorem \ref{thm-CritPtsExist} and Proposition \ref{PropCompactFVP}. By Lemma \ref{lemmFVPtoLVP}, \ref{ItCompactFVP} implies \ref{ItCompLVP}. Clearly, \ref{ItCompLVP} implies \ref{ItCompContLVP}, which by Proposition \ref{propLVPReverse} implies \ref{ItCompactFVP}.
\end{proof}

Other than compactness, if we impose at most one weakening to the $\lvp$ we obtain statements equivalent to $\pica$.

\begin{Theorem}\label{TheoCritEquivPiCA}
The following are equivalent:

\begin{enumerate}[label=(\alph*)]

\item\label{itPica} $\pica$;

\item\label{itFVP} the $\fvp$;

\item\label{itContLVP} the $\lvp$ for continuous $f$;

\item\label{itLVP} the $\lvp$.

\end{enumerate}
Moreover, in \ref{itFVP} and \ref{itLVP} we may take $\spc X$ to be the Baire space, and in any of \ref{itFVP}-\ref{itLVP} we may take $\spc X$ to be (the closed unit ball of) $\mathcal C \big ( [0,1] \big )$.
\end{Theorem}

\begin{proof}
The equivalence between \ref{itPica} and \ref{itFVP} is given by Theorem \ref{thm-CritPtsExist} and Proposition~\ref{prop:critical-vs-Pi11CA-RCA}.  That \ref{itFVP} implies \ref{itLVP} is Lemma \ref{lemmFVPtoLVP}. Clearly, \ref{itLVP} implies \ref{itContLVP}, and the latter implies \ref{itPica} by Proposition \ref{propLVPReverse} (with Remark \ref{RemUnitBall}).
\end{proof}

We have also considered variants of the variational principle where $f$ is bounded, but this restriction has not affected the proof-theoretic strength of the theorem; indeed, in each of these theorems we may additionally assume that $f$ is bounded.

\section*{Acknowledgments}
We thank our anonymous reviewers for their many helpful suggestions for improving the clarity of this work.  This project was partially supported by the \emph{Fonds voor Wetenschappelijk Onderzoek -- Vlaanderen} Pegasus program, a grant from the John Templeton Foundation (``A new dawn of intuitionism: mathematical and philosophical advances'' ID 60842), JSPS KAKENHI (grant numbers 16K17640 and 15H03634), JSPS Core-to-Core Program (A.~Advanced Research Networks) and  JAIST Research Grant 2018(Houga).  The opinions expressed in this work are those of the authors and do not necessarily reflect the views of the John Templeton Foundation.

\bibliographystyle{amsplain}
\bibliography{RMFixedPointTheorems}

\end{document}